\documentclass{cmslatex}
\usepackage{latexsym, amssymb, enumerate, amsmath}

\renewcommand{\theequation}{\arabic{section}.\arabic{equation}}

\def\bR {\mathbf{R}}
\def\bS {\mathbf{S}}

\def\cB {\mathcal{B}}

\def\cD {\mathcal{D}}

\def\cH {\mathcal{H}}

\def\cK {\mathcal{K}}
\def\cL {\mathcal{L}}

\def\cN {\mathcal{N}}

\def\cV {\mathcal{V}}

\def\a {{\alpha}}
\def\b {{\beta}}
\def\g {{\gamma}}
\def\Ga {{\Gamma}}
\def\de {{\delta}}
\def\eps {{\epsilon}}
\def\th {{\theta}}
\def\ka {{\kappa}}
\def\l {{\lambda}}

\def\si {{\sigma}}

\def\om {{\omega}}
\def\Om {{\Omega}}

\def\d {{\partial}}
\def\grad {{\nabla}}

\def\rstr {{\big |}}
\def\Rstr {{\Bigg |}}
\def\indc {{\bf 1}}

\def\la {\langle}
\def\ra {\rangle}
\def \La {\bigg\langle}
\def \Ra {\bigg\rangle}
\def \lA {\big\langle \! \! \big\langle}
\def \rA {\big\rangle \! \! \big\rangle}

\newcommand{\Div}{\operatorname{div}}

\newcommand{\Supp}{\operatorname{supp}}
\newcommand{\Supess}{\mathop{\operatorname{supess}}}
\newcommand{\Det}{\operatorname{det}}

\newcommand{\Ker}{\operatorname{Ker}}
\newcommand{\Ran}{\operatorname{Ran}}

\newcommand{\ba}{\begin{aligned}}
\newcommand{\ea}{\end{aligned}}

\newcommand{\be}{\begin{equation}}
\newcommand{\ee}{\end{equation}}

\newcommand{\lb}{\label}

\newtheorem{thm}{Theorem}[section]
\newtheorem{prop}[thm]{Proposition}
\newtheorem{lem}[thm]{Lemma}

\newtheorem{defn}[thm]{Definition}

\begin{document}

\title{The Diffusion Approximation\\ for the Linear Boltzmann Equation\\ with Vanishing Scattering Coefficient
\thanks{
}}
          
\author{Claude Bardos\thanks{Laboratoire Jacques-Louis Lions, BP187, 4 place Jussieu, 75252 Paris Cedex 05 France, (claude.bardos@gmail.com)}
\and
Etienne Bernard\thanks{Institut G\'eographique National, Laboratoire de Recherche en G\'eod\'esie, Universit\'e Paris Diderot, B\^atiment Lamarck A, 5 rue Thomas Mann, Case courrier 7071, 75205 Paris Cedex 13 France, (esteve.bernard@gmail.com)}
\and
Fran\c cois Golse\thanks{Ecole Polytechnique, Centre de Math\'ematiques Laurent Schwartz, 91128 Palaiseau Cedex, France, (francois.golse@math.polytechnique.fr)}
\and
R\'emi Sentis\thanks{Laboratoire Jacques-Louis Lions, BP187, 4 place Jussieu, 75252 Paris Cedex 05 France, (sentis.remi@gmail.com)}
}



\pagestyle{myheadings}\markboth{Diffusion Approximation with Vanishing Scattering}{C. Bardos, E. Bernard, F. Golse, R. Sentis}\maketitle

\begin{abstract}
The present paper discusses the diffusion approximation of the linear Boltzmann equation in cases where the collision frequency is not uniformly large in the spatial domain. Our results apply for 
instance to the case of radiative transfer in a composite medium with optically thin inclusions in an optically thick background medium. The equation governing the evolution of the approximate 
particle density coincides with the limit of the diffusion equation with infinite diffusion coefficient in the optically thin inclusions.
\end{abstract}


\rightline{\em Pour George Papanicolaou, en t\'emoignage de gratitude et d'amiti\'e.}

\section{Introduction}\label{S-Intro}


The linear Boltzmann equation is a kinetic model used in many different contexts. It appeared (perhaps for the first time) in a paper by Lorentz \cite{Lorentz05} on the motion of electrons in metals, 
and has been since then used in various branches of mathematical physics such as radiative transfer \cite{PomraningRT,MihalasRH}, neutron transport theory \cite{WeinWigner}. A typical model 
linear Boltzmann equation is
\be\lb{LinBoltz0}
(\d_t+\om\cdot\grad_x)f(t,x,\om)=\si\left(\tfrac1{4\pi}\int_{\bS^2}f(t,x,\om')ds(\om')-f(t,x,\om)\right)
\ee
where $t\ge 0$, $x\in\bR^3$ and $\om\in\bS^2$ are respectively the time, position and direction, while $ds$ is the surface element on $\bS^2$. The function $f\equiv f(t,x,\om)$ is the distribution 
function of a population of monokinetic particles (such as photons) moving at speed $1$ in a background medium (such as a planetary or stellar atmosphere). The coefficient $\si>0$ is the 
scattering rate in the medium. The simple model above assumes that the scattering mechanism is isotropic and that the absorption and scattering rate are equal. (The model above is essentially 
eq. (4.216) in \cite{PomraningRT}.) 

A classical approximation of solutions $f$ of the linear Boltzmann equation above is the so-called ``P1 appproximation'', where $f$ is replaced by its spherical harmonics expansion in the angle 
variable $\om$, truncated at order $1$ (see chapter IX in \cite{WeinWigner}). This approximation is also known as the ``Eddington approximation'' in radiative transfer (see for instance chapter III.2 
in \cite{PomraningRT}). The P1 approximation is
$$
f(t,x,\om)\simeq\rho(t,x)+3j(t,x)\cdot\om\,,
$$
so that
$$
\rho(t,x)\simeq\tfrac1{4\pi}\int_{\bS^2}f(t,x,\om')ds(\om')\,,\quad\hbox{ and }j(t,x)\simeq\tfrac1{4\pi}\int_{\bS^2}\om'f(t,x,\om')ds(\om')\,.
$$
This approximation is used in regimes where the scattering coefficient $\si\gg 1$. In other words, the diffusion approximation is justified when the mean free path of the particles between scattering 
events is much smaller than the typical length scale of the spatial domain. In this case $f\simeq\rho$ to leading order, so that $j=-\tfrac1{3\si}\grad_x\rho$. Thus, the P1 approximation becomes
\be\lb{P1Approx}
f(t,x,\om)\simeq\rho(t,x)-\tfrac1{\si}\om\cdot\grad_x\rho(t,x)
\ee
Averaging both sides of the linear Boltzmann equation in the variable $\om$ leads to the local conservation law of mass (or particle number)
$$
\d_t\int_{\bS^2}f(t,x,\om')\tfrac{ds(\om')}{4\pi}+\Div_x\int_{\bS^2}\om'f(t,x,\om')\tfrac{ds(\om')}{4\pi}=0\,.
$$
Substituting the P1 approximate expression for $f$ in the left hand side of this equality results in the diffusion equation
\be\lb{Diff0}
\d_t\rho-\Div_x\left(\tfrac1{3\si}\grad_x\rho\right)=0\,.
\ee
While the P1 or diffusion approximation of the linear Boltzmann equation has been used for a long time, its mathematical justification is more recent. A proof based on Hilbert's expansion 
\cite{LarsenKeller}, a formal expansion of the solution $f$ of the linear Boltzmann equation in powers of $1/\si$, can be found in \cite{BSS}; see also chapter XXI.5 in \cite{DL}. 

Other proofs are based on the representation of the solution $f$ of the linear Boltzmann equation in terms of stochastic process: see for instance \cite{PapanicoBAMS} and the references therein.  
For the original contributions to the subject, see \cite{PBenoist} and \cite{IlinHasmi}. 

The diffusion approximation (\ref{Diff0}) of the linear Boltzmann equation (\ref{LinBoltz0}) is also used for inhomogeneous media where the scattering rate $\si\equiv\si(x)$ varies smoothly with
the position variable. Yet the Cauchy problems for the linear Boltzmann equation (\ref{LinBoltz0}) and for the diffusion equation (\ref{Diff0}) are well posed, in some weak sense to be discussed 
below, under the only assumption that both $\si$ and $1/\si$ belong to $L^\infty(\bR^3)$. However both the proof of the diffusion approximation based on the Hilbert expansion \cite{BSS} and the 
proof based on stochastic processes \cite{PapanicoBAMS} require that the scattering rate $\si$ satisfy rather stringent smoothness assumptions. In particular, these smoothness assumptions 
exclude scattering rates that are discontinuous functions of the position variable. This is unfortunate, since discontinuities in the scattering rate $\si(x)$ appear in the case of inhomogeneous or 
composite materials.

Another related issue is the order of magnitude of the scattering rate. As recalled above, the diffusion approximation (\ref{Diff0}) of the linear Boltzmann equation (\ref{LinBoltz0}) is justified 
if the scattering rate $\si(x)$ depends smoothly on the space variable $x$ and is large uniformly in $x$. In some applications involving strongly inhomogeneous media, the order of magnitude 
of the scattering rate may vary considerably in the spatial domain. For instance, in the context of neutron transport and nuclear reactor design, the scattering cross-section for neutron collisions 
of non-fission type in the uranium oxide is about 100 times higher than in water. Another example is the case of nondiffusive objects embedded in a diffusive medium in the context of medical 
imaging: see \cite{Bal2002}, especially the examples given on p. 1678.

The present work addresses the following problem.

\smallskip
\noindent
{\sc Problem:} to extend the validity of the diffusion approximation to cases where the scattering coefficient $\sigma\equiv\si(x)$ is neither continuous nor large uniformly as $x$ runs through the 
spatial domain. 

\smallskip
The formulation of this problem obviously includes cases where there is no separation of scale in the size of the scattering coefficient in the diffusive and nondiffusive regions. Asymptotic
expansions ``\`a la Hilbert'' cannot be used in such generality. 

Our approach is based instead on an energy method similar to the one used earlier in the derivation of the Rosseland equation from the radiative transfer equation \cite{BGPS}, or of the 
drift-diffusion equation from the Boltzmann equation with Fermi-Dirac statistics \cite{GolPou}, or even in the derivation of incompressible fluid dynamics from the Boltzmann equation of the 
kinetic theory of gases \cite{BGL2}.

\smallskip
The main results in this paper were presented by one of us (C.B.) to the conference ``Recent developments in applied mathematics'' in honor of G. Papanicolaou. Our own work on the
various asymptotic theories of kinetic models, including the diffusion approximation of the linear Boltzmann equation, the Rosseland approximation of the radiative transfer equation or 
the hydrodynamic limits of the Boltzmann equation in the kinetic theory of gases, was strongly influenced by G. Papanicolaou's remarkable contributions \cite{PapanicoBAMS,BLP} to
this subject. We are happy to dedicate this paper to our friend and colleague G. Papanicolaou on the occasion of his 70th birthday.


\section{Presentation of the problem and main result}\label{S-Main}


\subsection{The linear Boltzmann equation and the scaling assumptions}\label{SS-LinBoltz}

Consider the linear Boltzmann equation
\be\lb{LinB}
(\d_t+v\cdot\grad_x)f(t,x,v)+\cL_xf(t,x,v)=0
\ee
for the unknown $f\equiv f(t,x,v)$ that is the distribution function for a system of identical point particles interacting with some background material. In other words, $f(t,x,v)$ is the number density of 
particles located at the position $x\in\Om$, with velocity $v\subset\bR^N$ at time $t\ge 0$. Henceforth, we assume that $\Om$ is a bounded domain of $\bR^N$ with $C^1$ boundary $\d\Om$, and
that $\Om$ is locally on one side of $\d\Om$.

The notation $\cL_x$ designates a linear integral operator acting on the $v$ variable in $f$, i.e.
\be\lb{Lx}
\cL_xf(t,x,v)=\int_{\bR^N}k(x,v,w)(f(t,x,v)-f(t,x,w))d\mu(w)
\ee
where $\mu$ is a Borel probability measure on $\bR^N$, while $k$ is a nonnegative function defined $\mu\otimes\mu$-a.e. on $\bR^N\times\bR^N$ . We assume that $k$ satisfies the semi-detailed 
balance condition
\be\lb{SemDet}
\int_{\bR^N}k(x,v,w)d\mu(w)=\int_{\bR^N}k(x,w,v)d\mu(w)
\ee
and introduce the notation
\be\lb{Defa}
a(x,v):=\int_{\bR^N}k(x,v,w)d\mu(w)
\ee
for the scattering  rate, so that the kernel $k(x,v,w)$, up to the multiplicative coefficient $1/a$, measures the probability of a transition from velocity $w$ to velocity $v$ for particles located at the position 
$x$. Eventually one has:
$$
\cL_xf(t,x,v)=a(x,v)f(t,x,v)-\cK_xf(t,x,v)
$$
where $\cK_x$ designates the integral operator
\be\lb{DefK}
\cK_xf(t,x,v):=\int_{\bR^N}k(x,v,w)f(t,x,w)d\mu(w)\,.
\ee
The semi-detailed balance assumption appears for instance in \cite{LL10} --- see formula (2.9) in \S 2. The assumptions on the transition kernel other than (\ref{SemDet}) used in our discussion 
are introduced later.

Denoting by $n_x$ the unit outward normal field at $x\in\d\Om$, we henceforth consider the outgoing, characteristic and incoming components of $\d\Om\times\bR^N$:
$$
\ba
\Ga_+:=\{(x,v)\in\d\Om\times\bR^N\,|\,v\cdot n_x>0\}\,,
\\
\Ga_0:=\{(x,v)\in\d\Om\times\bR^N\,|\,v\cdot n_x=0\}\,,
\\
\Ga_-:=\{(x,v)\in\d\Om\times\bR^N\,|\,v\cdot n_x<0\}\,.
\ea
$$
The linear Boltzmann equation is supplemented with the absorption boundary condition
\be\lb{AbsBC}
f(t,x,v)=0\,,\qquad (x,v)\in\Ga_-\,,\quad t>0\,.
\ee
(In other words, it is assumed that there are no particles entering the domain $\Om$.) This choice is made for the sake of simplicity; other boundary conditions will be discussed later.

We next introduce the scaling assumption pertaining to the diffusion approximation of the linear Boltzmann equation (\ref{LinB}). Set $L$ to be a length scale that measures the size of $\Om$ 
while $V$ is the average particle speed;  consider the time scale $T:=L/V$. The diffusion limit of (\ref{LinB}) is based on the assumption that the dimensionless quantity $Ta(x,v)$ is large. We
introduce a scaling parameter $0<\eps\ll 1$ and set
$$
\hat k_\eps(x,v,w):=\eps k(x,v,w)
$$
so that $\hat k_\eps(x,v,w)$ is of order unity. Accordingly, we define
$$
\hat a_\eps(x,v):=\eps a_\eps(x,v)\,,\quad\hat\cL_x=\eps\cL_x\,,\quad\hbox{ and }\hat\cK_x=\eps\cK_x\,.
$$
(For notational simplicity, we do not mention explicitly the dependence of $\cL_x$ and $\cK_x$ in $\eps$.) Assume further that 
\be\lb{MeanV=0}
\int_{\bR^N}vd\mu(v)=0
\ee
and that variations of order unity of the boundary data driving the solution of (\ref{LinB}) do not occur on time scales shorter than $T/\eps$. This is obvious for the boundary condition (\ref{AbsBC}); 
however, this assumption is crucial and needs to be satisfied for some more general boundary conditions. In that case, the solution $f$ of (\ref{LinB}) is sought in the form
$$
f(t,x,v)=\hat f_\eps(\eps t,x,v)
$$
with the notation $\hat t=\eps t$ for the rescaled time variable.  Thus (\ref{LinB}) takes the form
$$
\eps\d_{\hat t}f_\eps(\hat t,x,v)+v\cdot\grad_xf_\eps(\hat t,x,v)+\frac1\eps\hat\cL_xf_\eps(\hat t,x,v)=0\,.
$$
Henceforth we drop hats on rescaled variables and consider the initial-boundary value problem for the scaled linear Boltzmann equation
\be\lb{ScalLinB}
\left\{
\ba
{}&(\eps\d_t+v\cdot\grad_x)f_\eps(t,x,v)+\frac1\eps\cL_xf_\eps(t,x,v)=0\,,&&\quad x\in\Om\,,\,\,v\in\bR^N\,,\,\,t>0\,,
\\
&f_\eps(t,x,v)=0\,,&&\quad (x,v)\in\Ga_-\,,\,\,t>0\,,
\\
&f_\eps(0,x,v)=f^{in}(x,v)&&\quad x\in\Om\,,\,\,v\in\bR^N\,,
\ea
\right.
\ee
with
\be
\cL_xf(t,x,v)=a_\epsilon(x,v)f(t,x,v)-\int k_\epsilon(x,v,w) f(t,x,w)d\mu(w)\,,
\ee
in the limit as $\eps\to 0$.

\subsection{Statement of the diffusion approximation}\lb{SS-DiffApprox}

From now on, we assume that the probability measure satisfies (\ref{MeanV=0}) and
\be\lb{Mom2mu}
0<\int_{\bR^N}|v\cdot\xi|^2d\mu(v)<\infty\qquad\hbox{ for all }\xi\in\bR^N\setminus\{0\}\,.
\ee

Assume that the spatial domain $\Om=A\cup B$, where $A$ is open and $B$ is closed in $\bR^N$ (i.e. $B\cap\d\Om=\varnothing$), with finitely many connected components denoted $B_l$, 
for $l=1,\ldots,m$. We further assume that $B_l$ has piecewise $C^1$ boundary, that $B_l$ is locally on one side of its boundary $\d B_l$. Finally, we denote by $n_x$ the unit normal field at 
$x\in\d A$, oriented towards the exterior of $A$. 

We further assume that the scattering kernel $k_\eps$ in the linear Boltzmann equation is a $dxd\mu(v)d\mu(w)$-a.e. nonnegative measurable function on $\Om\times\bR^N\times\bR^N$ 
satisfying the following assumptions, in addition to (\ref{SemDet}):

\smallskip
\noindent
(H1) the absorption rate $a_\eps$ is uniformly small on $B\times\bR^N$ as $\eps\to 0$, i.e.
\be\lb{ato0onB}
\|a_\eps\|_{L^\infty(B\times\bR^N,dxd\mu)}\to 0\quad\hbox{ as }\eps\to 0\,;
\ee
(H2) the restriction of $k_\eps$ to $A\times\bR^N\times\bR^N$ is assumed to be independent of $\eps$ and denoted $k_A\equiv k_A(x,v,w)$; it satisfies
\be\lb{Hyp1/k}
C_K:=\Supess_{(x,v)\in A\times\bR^N}\int_{\bR^N}\left(k_A(x,v,w)+\frac1{k_A(x,v,w)}+\frac1{k_A(x,w,v)}\right)d\mu(w)<\infty\,.
\ee
We henceforth denote
\be\lb{DefaA}
a_A(x,v):=\int_{\bR^N}k_A(x,v,w)d\mu(w)\,,\quad\hbox{ for }dxd\mu(v)-\hbox{a.e. }(x,v)\in A\times\bR^N\,.
\ee

\smallskip
The diffusion approximation requires one additional assumption on the set $B$ where the scattering rate vanishes as $\eps\to 0$. A first possibility is to postulate a lower bound on the scattering 
rate on $B$ which is compatible with assumption (H1) as $\eps\to 0$:

\noindent
\smallskip
(H3) the restriction of $k_\eps$ to $A\times\bR^N\times\bR^N$ is assumed to satisfy the bound
\be\lb{Hyp1/kB}
\Supess_{(x,v)\in B\times\bR^N}\int_{\bR^N}\frac{d\mu(w)}{k_\eps(x,v,w)}=o(1/\eps^2)\quad\hbox{ as }\eps\to 0\,.
\ee

\smallskip
However, the diffusion approximation can be proved even for scattering rates of order $O(\eps^2)$, including the case $a_\epsilon=0$ corresponding to vacuum, at the expense of an ergodicity 
assumption on the free transport operator in each connected component $B_l$ of $B$. For each $l=1,\ldots,m$, denote by $\tau_l\equiv\tau_l(x,v)$ the forward exit time  from $B_l$ starting from 
the position $x$ with the velocity $v$; in other words
\be\lb{DefTau}
\tau_l(x,v):=\inf\{t>0\hbox{ s.t. }x+tv\in\d B_l\}\,.
\ee

Instead of condition (H3), one can assume that 

\smallskip
\noindent
(H4) the Borel probability measure $\mu$ satisfies $\mu(\{0\})=0$ and, for each $l=1,\ldots,m$ and each $g\in L^2(\d B_l)$,
\be\lb{Ergo}
\ba
g(x+\tau_l(x,v)v)=g(x)\hbox{ for }d\si(x)d\mu(v)-\hbox{a.e. }(x,v)\in\d B_l\times\bR^N&
\\
\Rightarrow g(x)=\frac1{|\d B_l|}\int_{\d B_l}g(y)d\si(y)\hbox{ for a.e. }x\in\d B_l&\,,
\ea
\ee
where $d\si$ is the surface element on $\d B_l$.

For instance, the assumption (H4) is satisfied if the measure $\mu$ is spherically symmetric and if $B_l$ is convex for each $l=1,\ldots,m$.

\smallskip
The main result in this paper is summarized in the following theorem.

\begin{thm}\lb{T-WDiffApprox}
Assume that the Borel probability measure $\mu$ satisfies (\ref{MeanV=0})-(\ref{Mom2mu}), while the scattering kernel $k_\eps$ satisfies (\ref{SemDet}), together with the conditions (H1)-(H2)
and at least one of the conditions (H3) or (H4). Let $f^{in}\in L^2(\Om\times\bR^N;dxd\mu)$.

\noindent
(a) For each $\eps>0$, the Cauchy problem (\ref{ScalLinB}) has a unique weak solution $f_\eps$ belonging to $C_b(\bR_+;L^2(\Om\times\bR^N;dxd\mu(v)))$.

\noindent
(b) For a.e. $x\in A$, there exists a unique $\bR^N$-valued vector field $b^*(x,\cdot)\in L^2(\bR^N,d\mu)$ such that
$$
\cL_x^*b^*(x,v)=v\qquad\hbox{ and }\int_{\bR^N}b^*(x,v)d\mu(v)=0\,.
$$
(c) The $M_N(\bR)$-valued matrix field $M$ defined by 
$$
M_{ij}(x):=\int_{\bR^N}b^*_i(x,v)v_jd\mu(v)\quad\hbox{ for a.e. }x\in A\hbox{ and all }i,j=1,\ldots,N
$$
satisfies 
$$
|M_{ij}(x)|\le 2C_K\|v_i\|_{L^2(\bR^N,d\mu)}\|v_j\|_{L^2(\bR^N,d\mu)}\quad\hbox{ for a.e. }x\in A\hbox{ and all }i,j=1,\ldots,N
$$
and
$$
\sum_{i,j=1}^NM_{ij}(x)\xi_i\xi_j\ge\frac{\b}{2C_k}|\xi|^2\hbox{ for all }\xi\in\bR^N\,,\quad\hbox{ for a.e. }x\in A\,,
$$
where $\b>0$ is the smallest eigenvalue of the real symmetric matrix $S$ defined by
\be\lb{DefS}
S_{ij}:=\int_{\bR^N}v_iv_jd\mu(v)\,.
\ee
(d) In the limit as $\eps\to 0$,
$$
f_\eps\to\rho\hbox{ weakly-* in }L^\infty(\bR_+;L^2(\Om\times\bR^N;dxd\mu))
$$
where $\rho$ is the unique weak solution of
\be\lb{LimPDE}
\left\{
\ba
{}&\d_t\rho(t,x)=\Div_x(M(x)\grad_x\rho(t,x))\,,&&\quad x\in A\,,\,\,t>0\,,
\\	
&\rho(t,x)=0\,,&&\quad x\in\d\Om\,,\,\,t>0\,,
\\
&\rho(t,x)=\rho_l(t)\,,&&\quad x\in\d B_l\,,\,\,\quad l=1,\ldots,m\,,\,\,t>0\,,
\\
&\dot{\rho}_l(t)=\frac1{|B_l|}\int_{\d B_l}\frac{\d\rho}{\d n_M}(t,x)d\si(x)\,,&&\quad l=1,\ldots,m\,,\,\,t>0\,,
\\ 
&\rho(0,x)=\rho^{in}(x)\,,&&\quad x\in\Om\,.
\ea
\right.
\ee
and where $\rho^{in}$ is given by the following formula:
$$
\rho^{in}(x)=\left\{\ba{}&\int_{\bR^N}f^{in}(x,v)d\mu(v)\quad&&\hbox{ for a.e. }x\in A\\&\frac1{|B_l|}\iint_{B_l\times\bR^N}f^{in}(y,v)dyd\mu(v)&&\hbox{ for a.e. }x\in B_l\ea\right.
$$
\end{thm}

\smallskip
In (\ref{LimPDE}), we have used the standard notation
$$
\frac{\d\rho}{\d n_M}(t,x):=\sum_{i,j=1}^NM_{ij}(x)n_{x,i}\d_{x_j}\rho(t,x)\,.
$$

\smallskip
The diffusion approximation stated in Theorem \ref{T-WDiffApprox} (d) can be strengthened as follows, provided that the initial condition $f^{in}$ is independent of $v$ and constant in each
one of the connected components $B_l$ of $B$. 

\begin{thm}\lb{T-SDiffApprox}
Assume that the Borel probability measure $\mu$ satisfies (\ref{MeanV=0})-(\ref{Mom2mu}), while the scattering kernel $k_\eps$ satisfies (\ref{SemDet}), together with the conditions (H1)-(H2). 
and at least one of the conditions (H3) or (H4). Let $\rho^{in}\in L^2(\Om)$ satisfy the condition
$$
\rho^{in}(x)=\frac1{|B_l|}\int_{B_l}\rho^{in}(y)dy\quad\hbox{ for a.e. }x\in B_l\,,\quad l=1,\ldots,m\,.
$$
Assume further that 
$$
\cL_xb^*(x,v)=v\quad\hbox{ for a.e. }(x,v)\in A\times \bR^N\,,
$$
where $b^*$ is the vector field defined in Theorem \ref{T-WDiffApprox} (b). Then 

\noindent
(a) for a.e. $x\in A$, the matrix $M(x)$ defined in Theorem \ref{T-WDiffApprox} is symmetric;

\noindent
(b) the solution $f_\eps$ of the Cauchy problem for the linear Boltzmann equation (\ref{ScalLinB}) satisfies
$$
f_\eps(t,\cdot,\cdot)\to\rho(t,\cdot)\hbox{ strongly in }L^2(\Om\times\bR^N;dxd\mu)\hbox{ for all }t\ge 0
$$
and
$$
\frac1\eps\left(f_\eps-\int_{\bR^N}f_\eps d\mu(v)\right)\to -b^*\cdot\grad_x\rho\hbox{ strongly in }L^2(\bR_+\times\Om\times\bR^N;dxd\mu)
$$
as $\eps\to 0$, where $\rho$ is the unique weak solution of the diffusion problem (\ref{LimPDE}).
\end{thm}

\smallskip
Notice that the strong diffusion limit theorem (Theorem \ref{T-SDiffApprox}) involves the condition $\cL_xb^*(x,v)=v$, implying that the diffusion matrix $M(x)$ is symmetric. On the contrary,
the weak diffusion limit theorem (Theorem \ref{T-WDiffApprox}) applies to situations where $M(x)$ may fail to be symmetric. At the time of this writing, we do not know whether strong
convergence in the diffusion limit can be obtained with this level of generality in cases where the diffusion $M(x)$ is not symmetric for a.e. $x\in A$. 

\subsection{Remarks on Theorem \ref{T-WDiffApprox}}\lb{SS-RmksW}

The class of collision integrals $\cL_x$ considered here is obviously more general than in \cite{BSS}. In \cite{BSS}, it is assumed that the measure $\mu$ is the uniform probability measure on 
the set $V$ of admissible velocities, that can be a ball, or a sphere, or a spherical annulus centered at the origin in $\bR^N$. The scattering kernel $k_\eps(x,v,w)$ is independent of $x$ and
$\eps$, and is of the form 
$$
k_\eps(x,v,w)=\si\ka(v,w)\,,
$$
where $\si>0$ and
$$
0\le\frac1C\le\ka(v,w)=\ka(w,v)\le C\hbox{ for a.e. }(v,w)\in V\times V
$$
and 
$$
\int_V\ka(v,w)dw=1\quad\hbox{ for a.e. }v\in V\,,
$$ 
for some positive constant $C$. Furthermore, the main result in \cite{BSS} assumes that
$$
\ka(Rv,Rv')=\ka(v,v')\hbox{ for a.e. }(v,v')\in V\times V\,,\hbox{ for all }R\in O_N(\bR)\,.
$$
Under this assumption, the vector field $b^*$ is of the form $b^*(x,v)=\b^*(|v|)v$ for some real-valued measurable function defined a.e. on $\bR_+$ (see Lemma 4.2.4 in \cite{AllaireGolse} or \cite{DesvilleFG}
for an analogous result in a more complex situation), and the diffusion matrix field is of the form $M(x)=mI$, where $m>0$ is a positive constant (see the formulas (43)-(44)
in \cite{BSS}).

Assumption (H1) is obviously satisfied if $k_\eps(x,v,w)=0$ for $dxd\mu(v)d\mu(w)$-a.e. $(x,v,w)\in B\times\bR^N\times\bR^N$, or if $k_\eps(x,v,w)=O(\eps)$ on $B$. However the assumption 
used in the present paper is obviously much more general. For instance, it is satisfied if one has $k_\eps(x,v,w)=O(|\ln\eps|^{-\g_l})$ on $B_l$ with $\g_l>0$ for each $l=1,\ldots,m$. The Hilbert 
expansion method used in \cite{BSS} does not apply to this situation, and therefore cannot be used on the problem considered here in its fullest generality. The treatment of nondiffusive
embedded objects in \cite{Bal2002} assumes that $k_\eps=O(\eps)$ in $B$ (see \cite{Bal2002} on p. 1683), a situation much less general than the one considered here, which can be treated 
with the Hilbert expansion method and leads to a diffusion system analogous to (\ref{LimPDE}).

Even in the nondegenerate case where $B=\varnothing$, observe that our assumptions on the transition kernel $k_\eps$ do not imply that the vector field $b^*$ in Theorem \ref{T-WDiffApprox}
(b) depends smoothly on $x$. This again excludes the possibility of using the Hilbert expansion as in \cite{BSS} to establish the validity of the diffusion limit. Accordingly, the diffusion matrix 
field $M$ defined in Theorem \ref{T-WDiffApprox} (c) is in general not even continuous. The classical interpretation of the diffusion equation with diffusion matrix $M$ in terms of the associated 
stochastic differential equation fails in such a case (see for instance section 5.1 and Remark 5.1.6 in \cite{SV}). 

Yet, even though the Hilbert expansion method cannot be used on the scaled linear Boltzmann equation (\ref{ScalLinB}) with the level of generality implied by assumptions (H1)-(H4), notice
that the second convergence statement in Theorem \ref{T-SDiffApprox} provides information analogous to the knowledge of the next to leading order term in Hilbert's expansion. (This is easily
seen for instance on formula (\ref{P1Approx}) in the special case of the radiative transfer equation with isotropic scattering.) At variance with the usual diffusion approximation theory, this
information is available in the diffusion region only (i.e. for $x\in A$).

In the special case where the diffusion matrix has a type I discontinuity across some smooth surface is equivalent to a transmission problem for two diffusion equations on each side of the 
discontinuity surface, with continuity of the solution and of the normal component of the current across the discontinuity surface. See for instance \cite{LionsJLMReal} on p. 107 or Lemma 1.1 
in \cite{DesvFGRicci} for a discussion of this well known issue.

If $B_l$ is convex for $l=1,\ldots,m$ and $\mu$ is of the form $d\mu(v)=r(|v|)dv$ or $\mu$ is the uniform probability measure on a sphere included in $\bR^N$ centered at the origin, the condition
(H4) is obviously satisfied. Indeed, for each $x,y\in\d B_l$, the segment $[x,y]$ is included in $B_l$, so that $g(x)=g(y)$ for a.e. $x,y\in\d B_l$.

But even when $B_l$ is convex, the condition (H4) may fail to be satisfied for some measures $\mu$. For instance, assume that $N=2$,  and take $B_l=\{x\in\bR^2\hbox{ s.t. }|x|\le 1\}$. Denote 
by $(e_1,e_2)$ the canonical basis of $\bR^2$, and let 
$$
\mu=\tfrac14(\de_{e_1}+\de_{-e_1}+\de_{e_2}+\de_{-e_2})\,.
$$
For $x=(x_1,x_2)\in\d B_l$, one has $\tau_l(x,\pm e_1)=2|x_1|$ and $\tau_l(x,\pm e_2)=2|x_2|$, so that 
$$
\ba
(-x_1,x_2)+\tau_l(x,e_1)e_1=(x_1,x_2)\,,\qquad(x_1,x_2)-\tau_l(x,e_1)e_1=(-x_1,x_2)\,,
\\
(x_1,-x_2)+\tau_l(x,e_2)e_2=(x_1,x_2)\,,\qquad(x_1,x_2)-\tau_l(x,e_2)e_2=(x_1,-x_2)\,.
\ea
$$
Thus $g(x)=|x_1|$ or $g(x)=|x_2|$ are not a.e. constant on $\d B_l$ and yet satisfy the condition
$$
g(x+\tau_l(x,v)v)=g(x)\quad\hbox{ for }dxd\mu(v)-\hbox{a.e. }(x,v)\in\d B_l\times\bR^2\,.
$$

Comparing Theorems \ref{T-WDiffApprox} and \ref{T-SDiffApprox} with the result in \cite{GJL2} is a more delicate issue. We recall that the problem considered in \cite{GJL2} involves the 
juxtaposition of a medium where the collision cross-section is of order $1$ and a highly collisional medium, where the collision cross section is of order $1/\eps$. The setting is one dimensional, 
but extensions to higher dimensions are possible and discussed in \cite{GJL2}. The main result in \cite{GJL2} is a proof of the validity of a domain decomposition strategy where the highly 
collisional medium is treated by the diffusion equation, with a boundary layer term that is the solution of a Milne problem (see \cite{BSS}) to accurately describe the interface. The interested reader 
is referred to \cite{GJL2} for a more accurate description of this domain decomposition algorithm. 

At first sight, the situation considered in the present paper is of the same type, as the case of a transition kernel $k_\eps$ such that $k_\eps(x,v,w)=O(1)$ for a.e. $x\in B$ is covered by our 
assumptions. Yet the result in Theorems \ref{T-WDiffApprox} and \ref{T-SDiffApprox}  obviously does not involve any sophisticated treatment of the interface between $A$ and $B$ that would 
require solving a Milne problem. The difference between both results comes from the type of boundary data considered in \cite{GJL2} and here. In the situation considered in Theorems
\ref{T-WDiffApprox} and \ref{T-SDiffApprox}, the distribution function of particles entering each connected component of $B$, i.e. of the region where the collision cross-section is of order $1$, 
is independent of the variable $v$. For such boundary data, one easily verifies that the boundary layer matching the kinetic and the diffusion domain in \cite{GJL2} is trivial to leading order.

Finally, notice that the time scale on which the evolution of the linear Boltzmann equation (\ref{ScalLinB}) is observed is the same in the diffusive as well as in the nondiffusive parts of $\Om$.
This may be slightly surprising at first sight, since the diffusion approximation involves a near equilibrium regime, and therefore needs to be observed on a time scale much longer than the 
original time scale for the linear Boltzmann equation. One of the difficulties in matching diffusive and nondiffusive regions in the theory of the linear Boltzmann equation is that diffusion and
transport are phenomena evolving on different time scales. However, in the situation considered here, the boundary of each one of the connected components $B_l$ of $B$ does not touch
$\d\Om$. Therefore, the transport process in each $B_l$ is driven by the surrounding diffusive region. This explains why our asymptotic theory in Theorems \ref{T-WDiffApprox} and 
\ref{T-SDiffApprox} involves the same time scale in the diffusive and in the nondiffusive regions.


\section{The collision integral and the diffusion matrix}\lb{S-CollInt}


\subsection{Properties of the integral operator $\cL_x$}\lb{SS-PropLx}

Henceforth we denote
$$
\la\phi\ra:=\int_{\bR^N}\phi(v)d\mu(v)\,,\quad\hbox{ for all }\phi\in L^1(\bR^N,d\mu)\,.
$$

\smallskip
\begin{lem}\lb{L-PropLx}
Assume that $\mu$ is a Borel probability measure on $\bR^N$, and that $k_\eps$ is a nonnegative $dxd(\mu\otimes\mu)$-measurable function defined $dxd(\mu\otimes\mu)$-a.e. on 
$\Om\times\bR^N\times\bR^N$ satisfying (\ref{SemDet}). Assume further that the function $a_\eps\equiv a_\eps(x,v)$ defined in (\ref{Defa}) satisfies the condition
\be\lb{aBdd}
a_\eps\in L^\infty(\Om\times\bR^N,dxd\mu)\,.
\ee
(a) The integral operators $\cL_x$ and $\cK_x$ are bounded on $L^2(\bR^N,\mu)$ for a.e. $x\in\Om$, with
$$
\|\cK_x\|_{\cL(L^2(\bR^N,\mu))}\le\|a_\eps(x,\cdot)\|_{L^\infty(\bR^N,\mu)}\,.
$$
(b) The adjoints of $\cK_x$ and $\cL_x$ are given by the formulas
$$
\left\{
\ba
{}&\cK^*_x\phi(v)=\int_{\bR^N}k_\eps(x,w,v)\phi(w)d\mu(w)\hbox{ and }
\\
&\cL^*_x\phi(v)=\int_{\bR^N}k_\eps(x,w,v)(\phi(v)-\phi(w))d\mu(w)
\ea
\right.
$$
for a.e. $x\in\Om$.

\noindent
(c) For each $\phi\in L^2(\bR^N;d\mu)$, and for a.e. $x\in\Om$,
$$
\la\phi\cL_x\phi\ra=\tfrac12\iint_{\bR^N\times\bR^N}k_\eps(x,v,w)(\phi(v)-\phi(w))^2d\mu(v)d\mu(w)\,.
$$
(d)  For a.e. $x\in\Om$,
$$
\{\hbox{ functions a.e. constant on }\bR^N\}\subset\bR\subset\Ker(\cL_x)\cap\Ker(\cL^*_x)\,;
$$
if in addition $k_\eps(x,v,w)>0$ for $d\mu(v)d\mu(w)$-a.e. $(v,w)\in\bR^N\times\bR^N$, then
$$
\Ker(\cL_x)=\Ker(\cL^*_x)=\{\hbox{ functions a.e. constant on }\bR^N\}=\bR\,.
$$
\end{lem}

\begin{proof}
Statement (a) follows from Schur's lemma (Lemma 18.1.12 in \cite{Horm3} or Lemma 1 in \S 2 of chapter XXI in \cite{DL}). The formula for $\cK_x$ in statement (b) and the inclusion in statement 
(d) are obvious.

As for the formula for $\cL^*_x$ in statement (b), observe that
$$
\ba
\cL_x\phi(v)+\cK_x\phi(v)&=\int_{\bR^N}k_\eps(x,v,w)\phi(v)d\mu(w)
\\
&=\int_{\bR^N}k_\eps(x,w,v)\phi(v)d\mu(w)=a_\eps(x,v)\phi(v)
\ea
$$
for $dxd\mu$-a.e. $(x,v)\in\Om\times\bR^N$ by the semi-detailed balance assumption (\ref{SemDet}).

For each $\phi\in L^2(\bR^N;d\mu)$ and a.e. in $x\in\Om$
$$
\ba
\la\phi\cL_x\phi\ra&=\iint_{\bR^N\times\bR^N}k_\eps(x,v,w)(\phi(v)^2-\phi(v)\phi(w))d\mu(v)d\mu(w)
\\
&=\int_{\bR^N}a_\eps(x,v)\phi(v)^2d\mu(v)-\iint_{\bR^N\times\bR^N}k_\eps(x,v,w)\phi(v)\phi(w)d\mu(v)d\mu(w)
\\
&=\tfrac12\int_{\bR^N}a_\eps(x,v)\phi(v)^2d\mu(v)+\tfrac12\int_{\bR^N}a_\eps(x,w)\phi(w)^2d\mu(v)
\\
&-\iint_{\bR^N\times\bR^N}k_\eps(x,v,w)\phi(v)\phi(w)d\mu(v)d\mu(w)
\\
&=\iint_{\bR^N\times\bR^N}k_\eps(x,v,w)\tfrac12(\phi(v)^2+\phi(w)^2)d\mu(v)d\mu(w)
\\
&-\iint_{\bR^N\times\bR^N}k_\eps(x,v,w)\phi(v)\phi(w)d\mu(v)d\mu(w)
\\
&=\tfrac12\iint_{\bR^N\times\bR^N}k_\eps(x,v,w)(\phi(v)-\phi(w))^2d\mu(v)d\mu(w)
\ea
$$
by Fubini's theorem and the semi-detailed balance assumption (\ref{SemDet}). This proves statement (c).

By statement (c), if $\phi\in L^2(\bR^N;d\mu)$ satisfies $\cL_x\phi=0$, then
$$
0=\la\phi\cL_x\phi\ra=\tfrac12\iint_{\bR^N\times\bR^N}k_\eps(x,v,w)(\phi(v)-\phi(w))^2d\mu(v)d\mu(w)\,.
$$
Therefore
$$
k_\eps(x,v,w)(\phi(v)-\phi(w))=0\quad\hbox{ for }d\mu(v)d\mu(w)-\hbox{ a.e. }(v,w)\in\bR^N\times\bR^N
$$
so that
$$
\phi(v)-\phi(w)=0\quad\hbox{ for }d\mu(v)d\mu(w)-\hbox{ a.e. }(v,w)\in\bR^N\times\bR^N\,.
$$
Averaging in $w$ shows that
$$
\phi(v)=\la\phi\ra\quad\hbox{ for }d\mu(v)-\hbox{ a.e. }v\in\bR^N\,,
$$
so that
$$
\Ker(\cL_x)\subset\{\hbox{ functions a.e. constant on }\bR^N\}=\bR\,.
$$
Since the function $(v,w)\mapsto k_\eps(x,w,v)$ satisfies the same properties as $k_\eps$, 
$$
\Ker(\cL_x)\subset\{\hbox{ functions a.e. constant on }\bR^N\}=\bR\,,
$$
and the proof is complete.
\end{proof}

\subsection{The Fredholm alternative: proof of Theorem \ref{T-WDiffApprox} (b)-(c)}\lb{SS-FredLx}

\begin{prop}\lb{P-FredLx}
Assume that the Borel probability measure $\mu$ satisfies (\ref{MeanV=0})-(\ref{Mom2mu}), and that $k_\eps$ is a nonnegative $dxd(\mu\otimes\mu)$-measurable function defined 
$dxd(\mu\otimes\mu)$-a.e. on $\Om\times\bR^N\times\bR^N$ satisfying (\ref{SemDet}) and (H2).

\smallskip
\noindent
(a) For a.e. $x\in A$ and each $\phi\in L^2(\bR^N,d\mu)$
$$
\ba
\|\phi-\la\phi\ra\|_{L^2(\bR^N;d\mu)}\le 2C_K\|\cL_x\phi\|_{L^2(\bR^N,d\mu)}\,,
\\
\|\phi-\la\phi\ra\|_{L^2(\bR^N;d\mu)}\le 2C_K\|\cL_x^*\phi\|_{L^2(\bR^N,d\mu)}\,.
\ea
$$
(b) For a.e. $x\in A$, the operators $\cL_x$ and $\cL_x^*$ are bounded Fredholm operators on $L^2(\bR^N;d\mu)$ with
$$
\Ran(\cL_x)=\bR^\bot\quad\hbox{ and }\Ran(\cL_x^*)=\bR^\bot\,.
$$

\noindent
(c) For a.e. $x\in A$, there exists unique $\bR^N$-valued vector fields $b(x,\cdot)$ and $b^*(x,\cdot)$ in $L^2(\bR^N,d\mu)$ such that
$$
\left\{
\ba
\cL_xb(x,v)&=v\qquad\hbox{ and }\la b(x,\cdot)\ra=0
\\
\cL_x^*b^*(x,v)&=v\qquad\hbox{ and }\la b^*(x,\cdot)\ra=0\,,
\ea
\right.
$$
(d) For a.e. $x\in A$ and all $i,j=1,\ldots,N$, one has
$$
\int_{\bR^N}b^*_i(x,v)v_jd\mu(v)=\int_{\bR^N}v_ib_j(x,v)d\mu(v)\,.
$$
\end{prop}

\smallskip
Notice that statement (b) in Theorem \ref{T-WDiffApprox} is exactly the part of statement (b) in Proposition \ref{P-FredLx} concerning the vector field $b^*$.

\begin{proof}
Set $\cL_x\phi=\psi$; by statement (c) in Lemma \ref{L-PropLx}
$$
\ba
\la\phi\psi\ra=\la\phi\cL_x\phi\ra=\tfrac12\iint_{\bR^N\times\bR^N}k_\eps(x,v,w)(\phi(v)-\phi(w))^2d\mu(v)d\mu(w)\ge 0\,.
\ea
$$
By the Cauchy-Schwarz inequality, for a.e. $x\in A$, 
$$
\ba
|\phi(v)-\la\phi\ra|^2&=\left(\int_{\bR^N}(\phi(v)-\phi(w))d\mu(w)\right)^2
\\
&\le\int_{\bR^N}\frac{d\mu(w)}{k_A(x,v,w)}\int_{\bR^N}k_A(x,v,w)(\phi(v)-\phi(w))^2d\mu(w)
\ea
$$
so that
$$
\|\phi-\la\phi\ra\|_{L^2(\bR^N;d\mu)}^2\le C_K\iint_{\bR^N\times\bR^N}k_A(x,v,w)(\phi(v)-\phi(w))^2d\mu(v)d\mu(w)=2C_K\la\phi\psi\ra\,.
$$
Next
$$
\la\psi\ra=\la\cL_x\phi\ra=\la(\cL^*_x1)\phi\ra=0
$$ 
since $\cL_x^*1=0$ by Lemma \ref{L-PropLx} (d), so that
$$
\la\phi\psi\ra=\la(\phi-\la\phi\ra)\psi\ra\le\|\psi\|_{L^2(\bR^N,d\mu)}\|\phi-\la\phi\ra\|_{L^2(\bR^N,d\mu)}
$$ 
by the Cauchy-Schwarz inequality. Putting together the last two inequalities, we obtain the bound
$$
\|\phi-\la\phi\ra\|_{L^2(\bR^N;d\mu)}\le 2C_K\|\psi\|_{L^2(\bR^N,d\mu)}\,.
$$
Since the function $k_A(x,w,v)$ satisfies the same assumptions as $k_A$, the analogous inequality for the adjoint operator $\cL^*_x$ immediatly follows, which proves statement (a). 

Pick $x\in A$ such that $a_A(x,\cdot)\in L^\infty(\bR^N;d\mu)$ and the inequalities in statement (a) are verified. By Lemma \ref{L-PropLx} (a)-(b), the operators $\cL_x$ and $\cL_x^*$ are 
bounded on $L^2(\bR^N;d\mu)$. Besides $\Ran(\cL_x)$ is closed, shown by the following argument. Let $\psi\in L^2(\bR^N,d\mu)$ and $\phi_n$ be a sequence of $L^2(\bR^N,d\mu)$ 
such that
$$
\cL_x\phi_n\to\psi\quad\hbox{ in }L^2(\bR^N,d\mu)\hbox{ as }n\to\infty\,.
$$
In particular $\cL_x\phi_n$ is a Cauchy sequence; since 
$$
\|(\phi_n-\la\phi_n\ra)-(\phi_m-\la\phi_m\ra)\|_{L^2(\bR^N,d\mu)}\le 2C_K\|\cL_x\phi_n-\cL_x\phi_m\|_{L^2(\bR^N,d\mu)}
$$
we conclude that $\phi_n-\la\phi_n\ra$ is a Cauchy sequence in $L^2(\bR^N,d\mu)$. Therefore
$$
\phi_n-\la\phi_n\ra\to\tilde{\phi}\quad\hbox{�in }L^2(\bR^N,d\mu)\hbox{ as }n\to\infty
$$
so that
$$
\cL_x(\phi_n-\la\phi_n\ra)=\cL_x\phi_n\to\cL_x\tilde\phi\quad\hbox{� in }L^2(\bR^N,d\mu)\hbox{ as }n\to\infty\,.
$$
Hence $\psi=\cL_x\tilde\phi\in\Ran(\cL_x)$. By the same token, $\Ran(\cL^*_x)$ is closed. Applying Corollary II.17 in \cite{Brezis} and Lemma \ref{L-PropLx} (d) shows that
$$
\Ran(\cL_x)^\bot=\Ker(\cL_x^*)=\bR\quad\hbox{ and }\Ran(\cL_x^*)^\bot=\Ker(\cL_x)=\bR\,.
$$
By Proposition II.12 in \cite{Brezis}, we conclude that $\cL_x$ and $\cL_x^*$ are Fredholm operators with
$$
\Ran(\cL_x)=\Ker(\cL_x^*)^\bot=\bR^\bot\quad\hbox{ and }\Ran(\cL_x^*)=\Ker(\cL_x)^\bot=\bR^\bot\,.
$$

The existence and uniqueness of the vector fields $b$ and $b^*$ in statement (c) follows from the orthogonality condition (\ref{MeanV=0}) and the Fredholm alternative in statement (b).

Finally, for a.e. $x\in A$ and all $i,j=1,\ldots,N$, one has
$$
\ba
\int_{\bR^N}b^*_i(x,v)v_jd\mu(v)&=\int_{\bR^N}b^*_i(x,v)\cL_xb_j(x,v)v_jd\mu(v)
\\
&=\int_{\bR^N}\cL^*_xb^*_i(x,v)b_j(x,v)v_jd\mu(v)=\int_{\bR^N}v_ib_j(x,v)d\mu(v)\,,
\ea
$$
which proves statement (d). 
\end{proof}

\smallskip
It remains to prove statement (c) in Theorem \ref{T-WDiffApprox}.

\begin{proof}[Proof of Theorem \ref{T-WDiffApprox} (c)]
The inequality in Proposition \ref{P-FredLx} (a) implies that
$$
\|b^*_i(x,\cdot)\|_{L^2(\bR^N,d\mu)}\le 2C_K\|\cL^*b^*_i(x,\cdot)\|_{L^2(\bR^N,d\mu)}=2C_K\|v_i\|_{L^2(\bR^N,d\mu)}
$$
for all $i=1,\ldots,N$ and a.e. $x\in A$. This bound implies the first inequality in statement (c) of Theorem \ref{T-WDiffApprox} as a consequence of the definition of the $M_{ij}(x)$ and of
the Cauchy-Schwarz inequality.

As for statement (c), for a.e. $x\in A$ and all $\xi\in\bR^N$,
$$
\ba
(\xi\cdot v)^2&=(\cL_x(\xi\cdot b(x,v)))^2
\\
&=\left(\int_{\bR^N}k_A(x,v,w)(\xi\cdot b(x,v)-\xi\cdot b(x,w))d\mu(w)\right)^2
\\
&\le a_A(x,v)\int_{\bR^N}k_A(x,v,w)(\xi\cdot b(x,v)-\xi\cdot b(x,w))^2d\mu(w)
\ea
$$
by the Cauchy-Schwarz inequality. On the other hand, by Lemma \ref{L-PropLx} (c), 
$$
\ba
\sum_{i,j=1}^NM_{ij}(x)\xi_i\xi_j&=\int_{\bR^N}(\xi\cdot b(x,v))\cL_x(\xi\cdot b(x,v))d\mu(v)
\\
&=\tfrac12\iint_{\bR^N\times\bR^N}k_A(x,v,w)(\xi\cdot b(x,v)-\xi\cdot b(x,w))^2d\mu(v)d\mu(w)\,.
\ea
$$
Hence
$$
\ba
\sum_{i,j=1}^NS_{ij}\xi_i\xi_j&=\int_{\bR^N}(\xi\cdot v)^2d\mu(v)
\\
&\le C_K\iint_{\bR^N\times\bR^N}k_A(x,v,w)(\xi\cdot b(x,v)-\xi\cdot b(x,w))^2d\mu(v)d\mu(w)
\\
&=2C_K\sum_{i,j=1}^NM_{ij}(x)\xi_i\xi_j\,.
\ea
$$
The conclusion follows from the definition of $\b>0$ which implies the inequality
$$
\sum_{i,j=1}^NS_{ij}\xi_i\xi_j\ge\b|\xi|^2\,.
$$
\end{proof}


\section{Existence and uniqueness theory for the linear Boltzmann equation and for the diffusion problem}\label{S-ProofLinB}


\subsection{Proof of Theorem \ref{T-WDiffApprox} (a)}\label{SS-ProofLinB}


Henceforth we denote
$$
\lA\Phi\rA:=\iint_{\bR^N\times\bR^N}\Phi(v,w)d\mu(v)d\mu(w)\,,\quad\hbox{ for all }\Phi\in L^1(\bR^N\times\bR^N,d(\mu\otimes\mu))\,.
$$

\smallskip
\begin{prop}\lb{P-ExUnLinB}
Assume that $k_\eps$ is a nonnegative measurable function defined $dxd(\mu\otimes\mu)$-a.e. on $\Om\times\bR^N\times\bR^N$ satisfying (\ref{SemDet}) and (\ref{aBdd}) with $a_\eps$ 
defined by (\ref{Defa}). For each $\eps>0$ and each $f^{in}\in L^2(\Om\times\bR^N;dxd\mu)$, there exists a unique weak solution of the initial-boundary value problem (\ref{ScalLinB}) in 
the space $C_b(\bR_+;L^2(\Om\times\bR^N;dxd\mu))$. This solution satisfies

\noindent
(a) the continuity equation in the sense of distributions on $\bR_+^*\times\Om$:
$$
\d_t\la f_\eps\ra+\Div_x\frac1\eps\la vf_\eps\ra=0\,;
$$
(b) the ``entropy inequality''
$$
\int_\Om\la f_\eps(t,x,\cdot)^2\ra dx+\int_0^t\int_\Om\lA k_\eps(x,\cdot,\cdot)q_\eps(s,x,\cdot,\cdot)^2\rA dxds\le\int_\Om\la f^{in}(x,\cdot)^2\ra dx
$$
for each $\eps>0$ and each $t\ge 0$, where
$$
q_\eps(t,x,v,w):=\frac1\eps(f_\eps(t,x,v)-f_\eps(t,x,w))\,.
$$
\end{prop}

\begin{proof}
The operator $\phi(x,v)\mapsto\cL_x\phi(x,v)$ is a bounded perturbation of the advection operator $-v\cdot\grad_x$ with absorbing boundary condition (\ref{AbsBC}) that is the generator of a 
strongly continuous contraction semigroup on $L^2(\Om\times\bR^N;dxd\mu)$. This implies the existence and uniqueness of the weak solution $f_\eps$ of the initial-boundary value problem 
(\ref{ScalLinB}) in the functional space $C_b(\bR_+;L^2(\Om\times\bR^N;dxd\mu))$.  Statement (a) follows from the inclusion $\bR\subset\Ker(\cL^*_x)$ in Lemma \ref{L-PropLx} (d), while 
statement (b) follows from Lemma \ref{L-PropLx} (c) and the usual energy estimate for the transport equation with source.
\end{proof}

Statement (a) in Theorem \ref{T-WDiffApprox} follows from Proposition \ref{P-ExUnLinB}.

\subsection{Existence and uniqueness theory for the diffusion problem (\ref{LimPDE})}\lb{SS-LaxMilgram}

Define
$$
\cH:=\left\{u\in L^2(\Om)\hbox{ s.t. }u(x)=\frac1{|B_l|}\int_{B_l}u(y)dy\hbox{ for a.e. }x\in B_l\,,\,\,l=1,\ldots,n\right\}\,,
$$
and
$$
\cV:=\cH\cap H^1_0(\Om)=\{u\in H^1_0(\Om)\hbox{ s.t. }\grad u(x)=0\hbox{ for a.e. }x\in B_l\,,\,\,l=1,\ldots,n\}\,.
$$

For each $\rho^{in}\in\cH$, consider the following variational problem
\be\lb{LimVP}
\left\{
\ba
{}&\rho\in C_b(\bR_+;\cH)\cap L^2(\bR_+;\cV)\,,\quad\d_t\rho\in L^2(\bR_+;\cV')\,,&&\hbox{ and }\rho\rstr_{t=0}=\rho^{in}\,,
\\
&\frac{d}{dt}\int_\Om\rho(t,x)w(x)dx+\int_A\grad w(x)\cdot M(x)\grad_x\rho(t,x)dx=0\,,&&\hbox{ for a.e. }t\ge 0\,,
\\
&\hbox{}&&\hbox{ for all }w\in\cV\,.
\ea
\right.
\ee
where $x\mapsto M(x)$ is an $M_N(\bR)$-valued measurable matrix field such that $M_{ij}\in L^\infty(A)$ for all $i,j=1,\ldots,N$.

\begin{lem}\lb{P-EqVP=PDE}
Assume that $x\mapsto M(x)$ is an $M_N(\bR)$-valued measurable matrix field on $A$ such that $M_{ij}\in L^\infty(A)$ for all $i,j=1,\ldots,N$. Let $\rho\in C([0,T];\cH)\cap L^2([0,T];\cV)$ with 
$\d_t\rho\in L^2([0,T];\cV')$. Then $\rho$ is a solution of the variational problem (\ref{LimVP}) if and only if
$$
\left\{
\ba
{}&\d_t\rho-\Div_x(M\grad_x\rho)=0&&\quad\hbox{ in }\cD'(\bR_+^*\times A)\,,
\\	
&\rho(t,\cdot)\rstr_{\d\Om}=0&&\quad\hbox{ in }L^2([0,T];H^{1/2}(\d\Om))\,,
\\
&\dot{\rho}_l=\frac1{|B_l|}\La \frac{\d\rho}{\d n_M},1\Ra_{H^{-1/2}(\d B_l),H^{1/2}(\d B_l)}&&\quad\hbox{ in }H^{-1}((0,T))\,,\,\,l=1,\ldots,m
\\ 
&\rho\rstr_{t=0}=\rho^{in}\,.
\ea
\right.
$$
\end{lem}

\begin{proof}
Specializing (\ref{LimVP}) to the case where $w\in C^\infty_c(A)$ is equivalent to 
$$
\d_t\rho-\Div_x(M\grad_x\rho)=0\quad\hbox{ in }\cD'(\bR_+^*\times A)\,.
$$
In particular, the (a.e. defined) vector field
$$
(0,\tau)\times A\ni(t,x)\mapsto(\rho(t,x),-M(x)\grad_x\rho(t,x))
$$
is divergence free in $(0,\tau)\times A$. Applying statement (b) in Lemma \ref{L-NormTr} shows that
$$
\ba
0=\frac{d}{dt}\int_\Om\rho(t,x)w(x)dx+\int_A\grad w(x)\cdot M(x)\grad_x\rho(t,x)dx
\\
=\frac{d}{dt}\int_A\rho(t,x)w(x)dx+\sum_{l=1}^m|B_l|w_l\dot{\rho}_l(t)+\int_A\grad w(x)\cdot M(x)\grad_x\rho(t,x)dx
\\
=\sum_{l=1}^mw_l\left(|B_l|\dot{\rho}_l(t)-\La\frac{\d\rho}{\d n_M}(t,\cdot)\Rstr_{\d B_l},1\Ra_{H^{-1/2}(\d B_l),H^{1/2}(\d B_l)}\right)
\ea
$$
for each $w\in\cV$, where
$$
w_l:=\frac1{|B_l|}\int_{B_l}w(y)dy\,,\quad l=1,\ldots,m\,.
$$
Since this is true for all $w\in\cV$, and therefore for all $(w_1,\ldots,w_m)\in\bR^m$, one concludes that
$$
|B_l|\dot{\rho}_i-\La\frac{\d\rho}{\d n_M}\Rstr_{\d B_l},1\Ra_{H^{-1/2}(\d B_l),H^{1/2}(\d B_l)}=0
$$
in $H^{-1}((0,\tau))$ for all $l=1,\ldots,m$, which is precisely the transmission condition on $\d B_l$. Finally, the Dirichlet condition on $\d\Om$ comes from the condition $\rho\in L^2(\bR_+;\cV)$ 
since $\cV\subset H^1_0(\Om)$.

Conversely, if $\rho\in C_b(\bR_+,\cH)\cap L^2(\bR_+,\cV)$ s.t. $\d_t\rho\in L^2(\bR_+,\cV')$ satisfies the initial condition and the diffusion equation in (\ref{LimPDE}) in the sense of distributions 
on $\bR_+^*\times A$, together with the transmission condition on $\d B_l$ for each $l=1,\ldots,m$, it follows from the identity above that $\rho$ must satisfy (\ref{LimVP}).
\end{proof}

This lemma justifies the following definition.

\begin{defn}\lb{D-DefWSolLimPDE}
For $\rho^{in}\in\cH$, a weak solution of the problem (\ref{LimPDE}) is a function $\rho\equiv\rho(t,x)$ such that $\rho\in C_b(\bR_+;\cH)\cap L^2(\bR_+;\cV)$ and $\d_t\rho\in L^2(\bR_+;\cV')$
which satisfies the variational formulation and the initial condition in (\ref{LimVP}).
\end{defn}

\smallskip
The existence and uniqueness theory for the limiting diffusion problem is summarized in the following proposition.

\begin{prop}\lb{P-ExUnDiff}
Assume that $x\mapsto M(x)$ is an $M_N(\bR)$-valued measurable matrix field on $A$ satisfying 
$$
\ba
M_{ij}\in L^\infty(A)\hbox{ for all }i,j=1,\ldots,N\,,\hbox{ and }\hbox{there exists }\a>0\hbox{ s.t.}&
\\
\xi\cdot M(x)\xi\ge\a|\xi|^2\hbox{ for a.e. }x\in A\hbox{ and all }\xi\in\bR^N&\,.
\ea
$$
For each $\rho^{in}\in\cH$, the diffusion problem (\ref{LimPDE}) has a unique weak solution. This solution satisfies the ``energy'' identity for each $t\ge 0$:
$$
\tfrac12\int_\Om\rho(t,x)^2dx+\int_0^t\int_A\grad_x\rho(s,x)\cdot M(x)\grad_x\rho(s,x)dxds=\tfrac12\int_\Om\rho^{in}(x)^2dx\,.
$$
\end{prop}

\begin{proof}
The existence and uniqueness of the solution of the variational problem (\ref{LimVP}) is a straightforward consequence of the Lions-Magenes Theorem X.9 in \cite{Brezis}, with the bilinear form 
$$
a(u,w):=\int_A\grad u(x)\cdot M(x)\grad w(x)dx\,,\quad u,w\in\cV\,.
$$
Indeed, this bilinear form satisfies the assumptions of the Lions-Magenes theorem, since the first inequality in Theorem \ref{T-WDiffApprox} (c) (already established in section \ref{SS-FredLx})
implies that
$$
|a(u,w)|\le 2C_K\la|v|^2\ra\|\grad u\|_{L^2(A)}\|\grad w\|_{L^2(A)}\le 2C_K\la|v|^2\ra\|u\|_{\cV}\|w\|_{\cV}\,,
$$
while the second inequality there implies that
$$
a(u,u)\ge\frac{\b}{2C_K}\|\grad u\|^2_{L^2(A)}=\frac{\b}{2C_K}\|\grad u\|^2_{L^2(\Om)}=\frac{\b}{2C_K}(\|u\|^2_{\cV}-\|u\|^2_{\cH})
$$
for each $u,w\in\cV$.

Consider next the linear functional
$$
L(t):\,\cV\ni w\mapsto\la\d_t\rho,w\ra_{\cV',\cV}+a(\rho(t,\cdot),w)
$$
defined for a.e. $t\ge 0$. Since $L(t)=0$ for a.e. $t\in\bR$, one has
$$
\la L(t),\rho(t,\cdot)\ra_{\cV',\cV}=0\quad\hbox{ for a.e. }t\ge 0\,,
$$
for each $w\in\cV$. By Lemma \ref{L-Lempp}, one has
$$
L(t)=0\hbox{ in }\cV'\hbox{ for a.e. }t\in\bR_+\,.
$$
In particular, for a.e. $s\ge 0$, one has
$$
0=\la L(s),\rho(s,\cdot)\ra_{\cV',\cV}=\la\d_t\rho(s,\cdot),\rho(s,\cdot)\ra_{\cV',\cV}+\int_A\grad_x\rho(s,x)\cdot M(x)\grad_x\rho(s,x)dx\,,
$$
and one concludes by integrating in $s\in[0,t]$ and applying Lemma \ref{L-LemEnerg} b).
\end{proof}


\section{Diffusion approximation: proof of Theorem \ref{T-WDiffApprox} (d)}\lb{S-ProofWDiff}


The proof is split in several steps.

\smallskip
\noindent
\textit{Step 1: uniform bounds and weak compactness.}

By the entropy inequality (statement (b) in Proposition \ref{P-ExUnLinB}), one has the bounds
\be\lb{UnifBdfq}
\left\{
\ba
{}&\|f_\eps(t,\cdot,\cdot)\|_{L^2(\Om\times\bR^N;dxd\mu)}\le\|\rho^{in}\|_{L^2(\Om)}\quad\hbox{ and }
\\
&\|\sqrt{k_\eps}q_\eps\|_{L^2(\bR_+\times\Om\times\bR^N\times\bR^N;dtdxd(\mu\otimes\mu)}\le\|\rho^{in}\|_{L^2(\Om)}
\ea
\right.
\ee
By the Banach-Alaoglu theorem, the families $f_\eps$ and $\sqrt{k_\eps}q_\eps$ are relatively compact in $L^\infty(\bR_+;L^2(\Om\times\bR^N;dxd\mu))$ weak-* and 
$L^2(\bR_+\times\Om\times\bR^N\times\bR^N;dtdxd(\mu\otimes\mu))$ weak respectively. Extracting subsequences if needed, one has
\be\lb{Wlimf}
f_\eps\to f\hbox{ in }L^\infty(\bR_+;L^2(\Om\times\bR^N;dxd\mu))\hbox{ weak-*}
\ee
while
\be\lb{WLimq}
\sqrt{k_\eps}q_\eps\to r\hbox{ in }L^2(\bR_+\times\Om\times\bR^N\times\bR^N;dtdxd(\mu\otimes\mu))\hbox{ weak.}
\ee
In particular
\be\lb{WLim2q}
q_\eps\to q\hbox{ in }L^2(\bR_+\times A\times\bR^N\times\bR^N;k_A(x,v,w)dtdxd(\mu\otimes\mu))\hbox{ weak,}
\ee
where
\be\lb{Defq}
q(t,x,v,w):=r(t,x,v,w)/\sqrt{k_A(x,v,w)}\,,
\ee
for $dtdxd\mu(vd\mu(w)$-a.e. $(t,x,v,w)\in \bR_+\times A\times\bR^N\times\bR^N$.

\smallskip
\noindent
\textit{Step 2: asymptotic form of the linear Boltzmann equation}

One has
$$
\ba
\frac1\eps\cL_xf_\eps(t,x,v)=\int_{\bR^N}\indc_A(x)k_A(x,v,w)q_\eps(t,x,v,w)d\mu(w)
\\
+\int_{\bR^N}\indc_B(x)k_\eps(x,v,w)q_\eps(t,x,v,w)d\mu(w)
\ea
$$
Since $(x,v,w)\mapsto\indc_A(x)$ belongs to $L^2(A\times\bR^N\times\bR^N;k(x,v,w)dxd(\mu\otimes\mu))$ by (\ref{Hyp1/k})
$$
\int_{\bR^N}\indc_A(x)k_A(x,v,w)q_\eps(t,x,v,w)d\mu(w)
\to
\int_{\bR^N}\indc_A(x)k_A(x,v,w)q(t,x,v,w)d\mu(w)
$$
in the weak topology of $L^2(\bR_+\times A\times\bR^N;dtdxd\mu)$ as $\eps\to 0$. On the other hand, the Cauchy-Schwarz inequality and (\ref{Hyp1/k}) imply that
$$
\ba
\left\|\int_{\bR^N}k_\eps(\cdot,\cdot,w)q_\eps(\cdot,\cdot,\cdot,w)d\mu(w)\right\|^2_{L^2(\bR_+\times B\times\bR^N;dtdxd\mu)}
\\
\le
\|a_\eps\|_{L^\infty(B\times\bR^N)}\iint_{\bR_+\times\Om}\lA k_\eps(x,\cdot,\cdot)q_\eps(t,x,\cdot,\cdot)^2\rA dtdx
\\
\le\|a_\eps\|_{L^\infty(B\times\bR^N)}\|\rho^{in}\|^2_{L^2(\Om)}\to 0
\ea
$$
as $\eps\to 0$, by (\ref{ato0onB}) and the entropy inequality in Proposition \ref{P-ExUnLinB}. Thus
\be\lb{LimLf/e}
\frac1\eps\cL_xf_\eps(t,x,v)\to\int_{\bR^N}\indc_A(x)k_A(x,v,w)q(t,x,v,w)d\mu(w)
\ee
in the weak topology of $L^2(\bR_+\times\Om\times\bR^N;dxd\mu)$ as $\eps\to 0$. Passing to the limit in the scaled Boltzmann equation (\ref{ScalLinB}) we see that
\be\lb{LimRegf}
\ba
v\cdot\grad_xf\in L^2(\bR_+\times\Om\times\bR^N,dtdxd\mu)\quad\hbox{ and }
\\
\int_{\bR^N}k_A(\cdot,\cdot,w)q(\cdot,\cdot,\cdot,w)d\mu(w)\in L^2(\bR_+\times A\times\bR^N,dtdxd\mu)\,,
\ea
\ee
while
\be\lb{LimBEq}
v\cdot\grad_xf(t,x,v)+\indc_A(x)\int_{\bR^N}k_A(x,v,w)q(t,x,v,w)d\mu(w)=0\,,
\ee
for $dtdxd\mu$-a.e. $(t,x,v)\in\bR_+\times\Om\times\bR^N$.

\smallskip
\noindent
\textit{Step 3: asymptotic form of $f_\eps$.}

Multiplying both sides of the scaled linear Boltzmann equation (\ref{ScalLinB}) by $\eps$ and passing to the limit in the sense of distributions as $\eps\to 0$, one finds that
$$
\cL_xf(t,x,v)=0\quad\hbox{ for a.e. }(t,x,v)\in\bR_+^*\times\Om\times\bR^N\,.
$$
By Lemma \ref{L-PropLx} (d), this implies that $f(t,x,v)$ is independent of $v$ for a.e. $x\in A$, i.e. is of the form
\be\lb{f(t,x)}
f(t,x,v)=\rho(t,x)\quad\hbox{ for a.e. }(t,x,v)\in\bR_+^*\times A\times\bR^N\,.
\ee
By (\ref{Wlimf}) and (\ref{LimRegf})
\be\lb{RegRhoA}
\rho\in L^\infty(\bR_+;L^2(A))\quad\hbox{ and }\grad_x\rho\in L^2(\bR_+\times A)\,,
\ee
since 
$$
(v\cdot\grad_xf)v=(v\otimes v)\cdot\grad_x\rho\in L^2(\bR_+\times A;L^1(\bR^N,d\mu))
$$
so that
$$
S\cdot\grad_x\rho=\la v\otimes v\ra\cdot\grad_x\rho\in L^2(\bR_+\times A)\,;
$$
one concludes since $\Det(S)\not=0$ by assumption (\ref{Mom2mu}).

In particular 
$$
\rho\rstr_{\d B_i}\in L^2([0,T];H^{1/2}(\d B_i))
$$
for each $T>0$ and each $i=1,\ldots,n$. 

In particular, the first condition in (\ref{LimRegf}) implies that $s\mapsto f(t,x+sv,v)$ is continuous in $s$ for $dtdxd\mu$-a.e. $(t,x,v)\in\bR_+\times\Om\times\bR^N$. Therefore, we 
deduce from (\ref{LimBEq}) and (\ref{f(t,x)}) that, for each $l=1,\ldots,m$
\be\lb{vgradfinBl}
\left\{
\ba
{}&v\cdot\grad_xf(t,x,v)=0\,,\quad x\in B_l\,,\,\,v\in\bR^N\,,\,\,t>0\,,
\\
&f(t,x,v)=\rho(t,x)\,,\quad x\in\d B_l\,,\,\,v\in\bR^N\,,\,\,t>0\,.
\ea
\right.
\ee
At this point, things are different according to whether (H3) or (H4 )holds. 

Under assumption (H3), applying the Cauchy-Schwarz inequality and the entropy inequality in Proposition \ref{P-ExUnLinB}, one finds that, 
$$
\ba
\int_0^\infty\int_B\|f_\eps(t,x,\cdot)-\la f_\eps\ra(t,x)\|^2_{L^2(\bR^N,d\mu)}dxdt
\\
=\int_0^\infty\int_B\int_{\bR^N}\left(\int_{\bR^N}(f_\eps(t,x,v)-f_\eps(t,x,w))d\mu(w)\right)^2d\mu(v)dxdt
\\
\le\Supess_{(x,v)\in B\times\bR^N}\int_{\bR^N}\frac{d\mu(w)}{k_\eps(x,v,w)}
\\
\times\int_0^\infty\int_B\iint_{\bR^N\times\bR^N}k_\eps(x,v,w)(f_\eps(t,x,v)-f_\eps(t,x,w))^2d\mu(v)d\mu(w)dxdt
\\
\le\eps^2\Supess_{(x,v)\in B\times\bR^N}\int_{\bR^N}\frac{d\mu(w)}{k_\eps(x,v,w)}\|f^{in}\|_{L^2(\Om\times\bR^N)}\to 0
\ea
$$
as $\eps\to 0$. The condition (H3) is used precisely at this point, in order to bound the right hand side of the last inequality above. Therefore 
$$
f(t,x,v)=\la f\ra(t,x)=:\rho(t,x)\quad\hbox{ for a.e. }(t,x,v)\in\bR_+\times B_l\times\bR^N\,,\quad l=1,\ldots,m\,.
$$
This implies that $v\cdot\grad_x\rho(t,x)=0$ for a.e. $(t,x,v)\in\bR_+\times B_l\times\bR^N$ and thus $S\grad_x\rho(t,x)=0$, by the first equality in (\ref{vgradfinBl}). This implies in turn that
$\grad_x\rho(t,x)=0$ for a.e. $(t,x)\in\bR_+\times B_l$ since the matrix $S=\la v\otimes v\ra$ is invertible by (\ref{Mom2mu}). Since $B_l$ is connected, we conclude that, for $l=1,\ldots,m$,
\be\lb{f=rhol(t)}
f(t,x,v)=\rho_l(t):=\frac1{|B_l|}\iint_{B_l\times\bR^N}f(t,x,v)dxd\mu(v)\,,
\ee
for $dtdxd\mu$-a.e. $(t,x,v)\in\bR_+\times B_l\times\bR^N$.

Under assumption (H4), observe that the first equality in (\ref{vgradfinBl}) implies that
$$
\frac{d}{ds}f(t,x+sv,v)=0\quad\hbox{ for all }s\hbox{ s.t. }x+sv\in B_l
$$
for $dtdxd\mu(v)$-a.e. $(t,x,v)\in\bR_+\times B_l\times\bR^N$. Thus, in view of the condition on $\mu$ in (H4), one concludes that
$$
\rho(t,x+\tau_l(x,v)v)=\rho(t,x)\hbox{ for }d\si(x)d\mu(v)-\hbox{ a.e. }(x,v)\in\d B_l\times\bR^N
$$
by solving the boundary value problem above by the method of characteristics. By assumption (\ref{Ergo})
$$
\rho(t,x)=\frac1{|\d B_l|}\int_{\d B_l}\rho(t,y)d\si(y)=:\rho_l(t)\quad\hbox{ for a.e. }x\in\d B_l\,,
$$
for a.e. $t\ge 0$. In other words, $\rho(t,\cdot)$ is a.e. equal to a constant on $\d B_l$. Solving again for $f$ along characteristics, we conclude that (\ref{f=rhol(t)}) holds under assumption (H4)
even if (H3) is not verified. 

Summarizing, we have proved that
\be\lb{RegRho}
\ba
{}&f(t,x,v)=\rho(t,x)\hbox{ for }dtdxd\mu-\hbox{a.e. }(t,x,v)\in\Om
\\
&\hbox{with }\rho\in L^\infty(\bR_+;\cH)\quad\hbox{ and }\grad_x\rho\in L^2(\bR_+\times\Om)\,.
\ea
\ee

\smallskip
\noindent
\textit{Step 4: Fourier's law and continuity equation}

Observe that the flux satisfies
\be\lb{FlaFlux}
\ba
\frac1\eps\la vf_\eps(t,x,\cdot)\ra&=\frac1\eps\la(\cL^*_xb(x,\cdot))f_\eps(t,x,\cdot)\ra=\La b(x,\cdot)\frac1\eps\cL_xf_\eps(t,x,\cdot)\Ra
\\
&=\iint_{\bR^N\times\bR^N}b(x,v)k(x,v,w)q_\eps(t,x,v,w)d\mu(v)d\mu(w)
\ea
\ee
for a.e. $(t,x)\in\bR_+\times A$ and for all $\eps>0$. Since $b\in L^\infty(A;L^2(\bR^N;d\mu))$ by Proposition \ref{P-FredLx} (c), the function $(x,v,w)\mapsto\sqrt{k_A(x,v,w)}b(x,v)$ belongs 
to the space $L^\infty(A;L^2(\bR^N\times\bR^N;d\mu(v)\mu(w)))$. Thus
\be\lb{LimFlux}
\ba
\frac1\eps\la vf_\eps(t,x,\cdot)\ra=\iint_{\bR^N\times\bR^N}b(x,v)k(x,v,w)q_\eps(t,x,v,w)d\mu(v)d\mu(w)
\\
\to\iint_{\bR^N\times\bR^N}b(x,v)k(x,v,w)q(t,x,v,w)d\mu(v)d\mu(w)
\\
=\la b(x,\cdot)v\cdot\grad_x\rho(t,x)\ra=M(x)\grad_x\rho(t,x)
\ea
\ee
in for the weak topology of $L^2(\bR_+\times A)$ as $\eps\to 0$, on account of (\ref{LimBEq}).

Therefore, for each $w\in\cV$, one has
$$
\frac{d}{dt}\int_\Om\la f_\eps(t,x,\cdot)\ra w(x)dx+\int_A\frac1\eps\la vf_\eps(t,x,\cdot)\ra\cdot\grad w(x)dx=0\,,
$$
(since $\grad w=0$ on $\mathring{B}$) and passing to the limit in each side of this identity as $\eps\to 0$ shows that
\be\lb{WFormDiffEq}
\frac{d}{dt}\int_\Om\rho(t,x)w(x)dx+\int_A\grad w(x)\cdot M(x)\grad_x\rho(t,x)dx=0
\ee
in the sense of distributions on $\bR_+^*$.

\smallskip
\noindent
\textit{Step 5: limiting initial condition}

By (\ref{FlaFlux}) and the Cauchy-Schwarz inequality
$$
\ba
\left\|\frac1\eps\la vf_\eps\ra\right\|^2_{L^2([0,T]\times A)}
&\le
\int_{\bR_+}\int_A\lA k_A(x,\cdot,\cdot)q_\eps(t,x,\cdot,\cdot)^2\rA dxdt
\\
&\times\Supess_{x\in A}\iint_{\bR^N\times\bR^N}k_A(x,v,w)|b(x,v)|^2d\mu(v)d\mu(w)
\\
&\le 8C_K^3\la|v|^2\ra\|\rho^{in}\|^2_{L^2(\Om)}
\ea
$$
using the entropy inequality in Proposition \ref{P-ExUnLinB} and Proposition \ref{P-FredLx} (c). Since
$$
\frac{d}{dt}\int_\Om\la f_\eps(t,x,\cdot)\ra w(x)dx=\int_A\frac1\eps\la vf_\eps(t,x,\cdot)\ra\cdot\grad w(x)dx
$$
for each $w\in\cV$, one has
\be\lb{BndDf/Dt}
\left\|\frac{d}{dt}\int_\Om\la f_\eps(\cdot,x,\cdot)\ra w(x)dx\right\|\le (2C_K)^{3/2}\la|v|^2\ra^{1/2}\|\rho^{in}\|_{L^2(\Om)}\|\grad w\|_{L^2(\Om)}\,.
\ee
Applying the Ascoli-Arzela theorem shows that, for each $w\in\cV$
\be\lb{UnifWCv<f>}
\int_\Om(\la f_\eps(t,x,\cdot)\ra-\rho(t,x))w(x)dx\to 0\hbox{ uniformly in }t\in[0,T]
\ee
for all $T>0$. In particular
$$
\iint_\Om f^{in}(x,v)w(x)dxd\mu(v)=\int_\Om\la f_\eps(0,x,\cdot)\ra w(x)dx\to\int_\Om\rho(0,x))w(x)dx
$$
as $\eps\to 0$. Since the test function $w$ is constant on $B_l$ for each $l=1,\ldots,m$, 
\be\lb{VarCondInDiff}
\ba
\int_\Om\rho(0,x)w(x)dx&=\iint_\Om f^{in}(x,v)w(x)dxd\mu(v)
\\
&=\int_\Om\rho^{in}(x)w(x)dx\quad\hbox{ for each }w\in\cV\,,
\ea
\ee
with $\rho^{in}$ defined by the formula in Theorem \ref{T-WDiffApprox}

Returning to (\ref{BndDf/Dt}), we have proved that $\d_t\la f_\eps\ra$ is bounded in $L^2(\bR_+,\cV')$ for each $T>0$, so that
\be\lb{RegDRho/Dt}
\d_t\rho\in L^2(\bR_+;\cV')\,.
\ee
Since $\rho\in L^\infty(\bR_+;\cH)\cap L^2([0,T];\cV)$ for each $T>0$ by (\ref{RegRho}), we conclude from (\ref{RegDRho/Dt}) that
$$
\rho\in C_b(\bR_+;\cH)
$$
so that (\ref{VarCondInDiff}) implies that $\rho$ satisfies the initial condition in (\ref{LimPDE}).

\smallskip
\noindent
\textit{Step 6: Dirichlet condition}

Next we establish the Dirichlet condition on $\d\Om$ for the diffusion equation. The scaled linear Boltzmann equation implies that, for each $\chi\in C^1_c(\bR_+^*)$,
$$
v\cdot\grad_x\int_0^\infty\chi(t)f_\eps(t,x,v)dt=-\int_0^\infty\chi(t)\frac1\eps\cL_xf_\eps(t,x,v)dt+\eps\int_0^\infty\chi'(t)f_\eps(t,x,v)dt
$$
is bounded in $L^2(\Om\times\bR^N;dxd\mu)$ by (\ref{LimLf/e}), the uniform boundedness principle (the Banach-Steinhaus theorem) and the entropy inequality in Proposition \ref{P-ExUnLinB}, 
while
$$
\int_0^\infty\chi(t)f_\eps(t,x,v)dt
$$
is bounded in $L^2(\Om\times\bR^N;dxd\mu)$ by the same entropy inequality. Hence 
$$
0=\int_0^\infty\chi(t)f_\eps(t,\cdot,\cdot)dt\Rstr_{\Ga^-}\to\int_0^\infty\chi(t)\rho(t,\cdot)dt\Rstr_{\Ga^-}
$$
in $L^2(\Ga^-;|v\cdot n_x|\tau(x,v)\wedge 1d\si(x)dv)$ by Cessenat's trace theorem \cite{Cess1}, where the notation $\tau(x,v)$ designates the forward exit time from $\Om$ starting from $x$ 
with velocity $v$, i.e.
$$
\tau(x,v):=\inf\{t>0\hbox{ s.t. }x+tv\in\d\Om\}\,,\quad x\in\Om\,,\,\,v\in\bR^N\,.
$$
In particular
$$
\int_0^\infty\chi(t)\rho(t,\cdot)dt\Rstr_{\d\Om}=0\,.
$$
By (\ref{RegRho}), we already know that the limiting density $\rho\in L^2([0,T];H^1(\Om))$. Therefore
\be\lb{DirCond}
\rho(t,\cdot)\rstr_{\d\Om}=0\quad\hbox{ in }L^2([0,T];H^{1/2}(\d\Om))
\ee
for each $T>0$.

\smallskip
\noindent
\textit{Step 7: convergence to the diffusion equation}

Summarizing, we have proved that
$$
f_\eps\hbox{ is relatively compact in }L^\infty(\bR_+;L^2(\Om\times\bR^N,dxd\mu(v)))\hbox{ weak-*}
$$
and that, if $f$ is a limit point of $f_\eps$ as $\eps\to 0$, it is of the form 
$$
f(t,x,v)=\rho(t,x)\quad dtdxd\mu(v)-\hbox{a.e. in }(t,x,v)\in\bR_+\times\Om\times\bR^N
$$
where
\be\lb{RhoSpace}
\ba
{}&\rho\in L^\infty(\bR_+;\cH)\cap L^2(\bR_+;H^1_0(\Om))=L^\infty(\bR_+;\cH)\cap L^2(\bR_+;\cV)
\\
&\hbox{and }\d_t\rho\in L^2(\bR_+;\cV')\
\ea
\ee
since $\rho$ satisfies the Dirichlet boundary condition (\ref{DirCond}) and $\grad_x\rho\in L^2(\bR_+\times\Om)$ by (\ref{RegRho}), together with 
(\ref{RegDRho/Dt}). In particular, this implies that
\be\lb{ContRho}
\rho\in C_b(\bR_+;\cH)\,.
\ee
Besides $\rho$ satisfies (\ref{WFormDiffEq}) for each test function $w\in\cV$, together with the initial condition (\ref{VarCondInDiff}). Therefore $\rho$ is the unique solution of the Dirichlet 
problem for the diffusion equation with diffusion matrix $M(x)$ defined in Theorem \ref{T-WDiffApprox} (c) with infinite diffusivity in $B$, with initial data $\rho^{in}$. By compactness and 
uniqueness of the limit point, we conclude that
$$
f_\eps\to\rho\quad\hbox{ in }L^\infty(\bR_+;L^2(\Om\times\bR^N,dxd\mu ))\hbox{ weak-*}
$$
as $\eps\to 0$.

\subsection{Proof of Theorem \ref{T-SDiffApprox}}

Since $\cL_xb^*(x,v)=v$ for $dxd\mu$-a.e. $(x,v)\in A\times\bR^N$, we conclude from the uniqueness of the vector field $b$ in statement (c) of Proposition \ref{P-FredLx} that 
$b(x,v)=b^*(x,v)$ for $dxd\mu$-a.e. $(x,v)\in A\times\bR^N$. The identity in Proposition \ref{P-FredLx} (d) shows that
$$
M_{ij}(x)=\int_{\bR^N}b_i(x,v)v_jd\mu(v)=\int_{\bR^N}v_ib_j(x,v)d\mu(v)=M_{ji}(x)
$$
for a.e. $x\in A$ and all $i,j=1,\ldots,N$. This proves statement (a) in Theorem \ref{T-SDiffApprox}. 

It remains to prove the strong convergence in statement (b). The proof is based on the weak convergence already established in Theorem \ref{T-WDiffApprox} and on a squeezing argument
based on the entropy inequality for (\ref{ScalLinB}) and on the energy identity in (\ref{LimPDE}).

\textit{Step 1: limiting entropy production}

By definition of $q_\eps$, one has 
$$
q_\eps(t,x,v,w)=-q_\eps(t,x,w,v)
$$ 
for $dtdxd\mu(v)d\mu(w)$-a.e. $(t,x,v,w)\in\bR_+\times A\times\bR^N\times\bR^N$ and each $\eps>0$; by passing to the limit as $\eps\to 0$
$$
q(t,x,v,w)=-q(t,x,w,v)
$$
for $dtdxd\mu(v)d\mu(w)$-a.e. $(t,x,v,w)\in\bR_+\times A\times\bR^N\times\bR^N$. Defining
$$
k_A^s(t,x,v,w)=\tfrac12(k_A(t,x,v,w)+k_A(t,x,w,v))
$$
one has
$$
\lA k_A(x,\cdot,\cdot)q(t,x,\cdot,\cdot)^2\rA=\lA k_A^s(x,\cdot,\cdot)q(t,x,\cdot,\cdot)^2\rA
$$
for a.e. $(t,x)\in\bR_+\times A$. Likewise
$$
\ba
\iint_{\bR^N\times\bR^N}k_A(x,v,w)(\phi(v)-\phi(w))^2d\mu(v)d\mu(w)
\\
=\iint_{\bR^N\times\bR^N}k_A^s(x,v,w)(\phi(v)-\phi(w))^2d\mu(v)d\mu(w)
\ea
$$
and
$$
\ba
\iint_{\bR^N\times\bR^N}k_A(x,v,w)(\phi(v)-\phi(w))q(t,x,v,w)d\mu(v)d\mu(w)
\\
=\iint_{\bR^N\times\bR^N}k_A^s(x,v,w)(\phi(v)-\phi(w))q(t,x,v,w)d\mu(v)d\mu(w)
\ea
$$
for a.e. $(t,x)\in\bR_+\times\Om$. With $\phi(v)=\xi\cdot b(x,v)$ for some $\xi\in\bR^N$ to be chosen later, and applying the Cauchy-Schwarz inequality, one finds that
\be\lb{CSIneqEntrProd}
\ba
\left(\iint_{\bR^N\times\bR^N}k_A^s(x,v,w)\xi\cdot(b(x,v)-b(x,w))q(t,x,v,w)d\mu(v)d\mu(w)\right)^2
\\
\le
\iint_{\bR^N\times\bR^N}k_A(x,v,w)(\xi\cdot(b(x,v)-b(x,w)))^2d\mu(v)d\mu(w)\lA k_A(x,\cdot,\cdot)q(t,x,\cdot,\cdot)^2\rA\,.
\ea
\ee
On the other hand, by definition of $k_A^s$
$$
\ba
\iint_{\bR^N\times\bR^N}k_A^s(x,v,w)\xi\cdot(b(x,v)-b(x,w))q_\eps(t,x,v,w)d\mu(v)d\mu(w)
\\
=
\frac2\eps\iint_{\bR^N\times\bR^N}k_A^s(x,v,w)\xi\cdot(b(x,v)-b(x,w))f_\eps(t,x,v)d\mu(v)d\mu(w)
\\
=
\frac1\eps\la f_\eps(t,x,\cdot)(\cL_x+\cL^*_x)\xi\cdot b(x,\cdot)\ra 
=
\frac2\eps\la\xi\cdot vf_\eps(t,x,\cdot)\ra
\ea
$$
for a.e. $(t,x)\in\bR_+\times A$ where the last equality follows from the definitions of the vector fields $b$ and $b^*$ in Proposition \ref{P-FredLx} (c). Passing to the limit as $\eps\to 0$, 
one finds that
\be\lb{<kbq>}
\ba
\iint_{\bR^N\times\bR^N}k_A^s(x,v,w)\xi\cdot(b(x,v)-b(x,w))q(t,x,v,w)d\mu(v)d\mu(w)
\\
=-2\xi\cdot M(x)\grad_x\rho(t,x)
\ea
\ee
for a.e. $(t,x)\in\bR_+\times A$. On the other hand
\be\lb{<kqq>}
\ba
\iint_{\bR^N\times\bR^N}k_A(x,v,w)&(\xi\cdot(b(x,v)-b(x,w)))^2d\mu(v)d\mu(w)
\\
&=2\la\xi\cdot b(x,\cdot)\cL_x(\xi\cdot b(x,\cdot))\ra=2\la\xi\cdot b(x,\cdot)\xi\cdot v\ra=2\xi\cdot M(x)\xi
\ea
\ee
for a.e. $x\in A$, by Lemma \ref{L-PropLx} (c). Applying the Cauchy-Schwarz inequality in (\ref{<kbq>}) and using (\ref{<kqq>}) implies that
$$
2(\xi\cdot M(x)\grad_x\rho(t,x))^2\le\xi\cdot M(x)\xi\lA k_A(x,\cdot,\cdot)q(t,x,\cdot,\cdot)^2\rA\,.
$$
With the choice $\xi=\grad_x\rho(t,x)$, we conclude that
\be\lb{EntrProd>}
2\grad_x\rho(t,x)\cdot M(x)\grad_x\rho(t,x)\le\lA k_A(x,\cdot,\cdot)q(t,x,\cdot,\cdot)^2\rA
\ee
for a.e. $(t,x)\in\bR_+\times A$. 

\smallskip
\noindent
\textit{Step 2: strong convergence}

By Proposition \ref{P-ExUnLinB} (b), for each $t\ge 0$ and each $\eps>0$, one has
$$
\int_\Om\la f_\eps(t,x,\cdot)^2\ra dx+\int_0^t\int_\Om\lA k_\eps(x,\cdot,\cdot)q_\eps(s,x,\cdot,\cdot)^2\rA dxds\le\int_\Om\rho^{in}(x)^2dx\,.
$$
By Jensen's inequality
$$
\int_\Om\la f_\eps(t,x,\cdot)^2\ra dx\ge\int_\Om\la f_\eps(t,x,\cdot)\ra^2dx
$$
while, by convexity and weak convergence, 
$$
\varliminf_{\eps\to 0}\int_\Om\la f_\eps(t,x,\cdot)\ra^2dx\ge\int_\Om\rho(t,x)^2dx
$$
uniformly in $t\in[0,T]$ for each $T>0$ by (\ref{UnifWCv<f>}). By the same token, for each $T\in[0,\infty]$,
\be\lb{LimEntrProd>}
\varliminf_{\eps\to 0}\int_0^T\int_A\lA k_A(x,\cdot,\cdot)q_\eps(t,x,\cdot,\cdot)^2\rA dxdt\ge\int_0^T\int_A\lA k_A(x,\cdot,\cdot)q(t,x,\cdot,\cdot)^2\rA dxdt\,.
\ee

Since the weak solution $\rho$ of the diffusion problem (\ref{LimPDE}) satisfies 
$$
\int_\Om\rho(t,x)^2dx+2\int_0^t\int_A\grad_x\rho(s,x)\cdot M(x)\grad_x\rho(s,x)dxds=\int_\Om\rho^{in}(x)^2dx
$$
for each $t>0$ by Proposition \ref{P-ExUnDiff}, we conclude that
$$
\int_\Om\la f_\eps(t,x,\cdot)^2\ra dx\to\int_\Om\rho(t,x)^2dx\hbox{ for each }t\ge 0
$$
while
$$
\ba
\int_0^\infty\int_A\lA k_A(x,\cdot,\cdot)q_\eps(t,x,\cdot,\cdot)^2\rA dxdt&\to\int_0^\infty\int_A\lA k_A(x,\cdot,\cdot)q_\eps(t,x,\cdot,\cdot)^2\rA dxdt
\\
&=\int_0^\infty\int_A\grad_x\rho(t,x)\cdot M(x)\grad_x\rho(t,x)dxdt
\ea
$$
as $\eps\to 0$ by (\ref{EntrProd>}).

Therefore
$$
f_\eps(t,\cdot,\cdot)\to\rho\hbox{ strongly in }L^2(\Om\times\bR^N,dxd\mu)\hbox{ for all }t\ge 0
$$
and
$$
q_\eps\to q\hbox{ strongly in }L^2(\bR_+\times A\times\bR^N\times\bR^N,dtdxd(\mu\otimes\mu))
$$
as $\eps\to 0$ and (\ref{EntrProd>}) is an equality. In other words, for a.e. $(t,x)\in\bR_+\times A$, one has
$$
2\grad_x\rho(t,x)\cdot M(x)\grad_x\rho(t,x)=\lA k_A(x,\cdot,\cdot)q(t,x,\cdot,\cdot)^2\rA\,.
$$
Therefore $q$ satisfies the equality in the Cauchy-Schwarz inequality (\ref{CSIneqEntrProd}) (with $\xi=\grad_x\rho(t,x)$). This imples that $q$ is of the form
$$
q(t,x,v,w)=\l(t,x)(b^*(x,v)-b^*(x,w))\cdot\grad_x\rho(t,x)\,,
$$
for some measurable function $\l$ defined a.e. on $\bR_+\times A$. Inserting this expression for $q$ in (\ref{LimBEq}) we find that
$$
\ba
v\cdot\grad_x\rho(t,x)&=-\int_{\bR^N}q(t,x,v,w)d\mu(w)
\\
&=-\l(t,x)\grad_x\rho(t,x)\cdot\cL_xb^*(x,v)=-\l(t,x)v\cdot\grad_x\rho(t,x)\,,
\ea
$$
for a.e. $(t,x,v)\in\bR_+\times A\times\bR^N$. Therefore $\l(t,x)=-1$ for a.e. $(t,x)\in\bR_+\times A$ such that $\grad_x\rho(t,x)\not=0$, so that
$$
q(t,x,v,w)=-(b^*(x,v)-b^*(x,w))\cdot\grad_x\rho(t,x)\,,
$$
for a.e. $(t,x,v,w)\in\bR_+\times A\times\bR^N\times\bR^N$. Averaging in $w$, one finds that
$$
\ba
\frac1\eps\left(f_\eps(t,x,v)-\int_{\bR^N}f_\eps(t,x,w)d\mu(w)\right)&=\int_{\bR^N}q_\eps(t,x,v,w)d\mu(w)
\\
&\to\int_{\bR^N}q_\eps(t,x,v,w)d\mu(w)=-b^*(x,v)\cdot\grad_x\rho(t,x)
\ea
$$
in the strong topology of $L^2(\bR_+\times A\times\bR^N,dtdxd\mu)$ as $\eps\to 0$.

\section{Conclusions}

The main result presented above (Theorems \ref{T-WDiffApprox} and \ref{T-SDiffApprox}) can be generalized in several directions. 

First, our method obviously applies to a scaled linear Boltzmann equation of the form
$$
(\eps\d_t+v\cdot\grad_x)f_\eps(t,x,v)+\frac1\eps\cL_xf_\eps(t,x,v)+\eps\cB f_\eps(t,x,v)=\eps S(t,x,v)
$$
where $\cB$ is a bounded operator on $L^2(\Om\times\bR^N;dxd\mu)$, while the source term $S$ is chosen so that $S\in L^1(\bR_+;L^2(\Om\times\bR^N,dxd\mu))$. For instance $\cB$ could be 
the multiplication by an amplifying or damping coefficient, i.e. $\cB f_\eps(t,x,v)=\g(x)f_\eps(t,x,v)$ as in \cite{BSS}. In other words, problems where the collision process is nearly, but not exactly 
conservative can be treated exactly as above.

For some applications, for instance in the context of neutron transport theory, it would be important to extend the validity of the results presented in this paper to the case of scattering kernels which
fail to satisfy the semi-detailed balance condition (\ref{SemDet}).

More general boundary conditions than the absorbing condition on $\d\Om$ can also be considered. For instance, imposing a specular or diffuse reflection condition at the boundary, or a convex 
combination thereof, i.e. assuming that
$$
f_\eps(t,x,v)=(1-\th(x))f_\eps(t,x,v-2v\cdot n_xn_x)+\frac{\th(x)}{\la(w\cdot n_x)_+\ra}\int_{\bR^N}f_\eps(t,x,w)(w\cdot n_x)_+d\mu(w)
$$
with $\th\in C(\d\Om)$ satisfying $0\le\th(x)\le 1$ for all $x\in\d\Om$ and with a measure $\mu$ invariant under all transformations of the form $v\mapsto Qv$ for $Q\in O_N(\bR)$ leads to the same 
result as in Theorems \ref{T-WDiffApprox}-\ref{T-SDiffApprox}, except that the homogeneous Dirichlet condition on $\d\Om$ should be replaced with the homogeneous Neuman condition. 

Finally, since the compactness method used in the proof of Theorems \ref{T-WDiffApprox}-\ref{T-SDiffApprox}, finds its origin in \cite{BGPS}, we expect that the methods presented in this paper should 
also apply to some nonlinear problems, such as the radiative transfer equations.


\begin{appendix}

\section{Auxiliary Lemmas on Evolution Equations}


Let $\cV$ and $\cH$ be two separable Hilbert spaces such that $\cV\subset\cH$ with continuous inclusion and $\cV$ is dense in $\cH$. The Hilbert space $\cH$ is identified with its dual and the 
map 
$$
\cH\ni u\mapsto L_u\in\cV'\,,
$$
where $L_u$ is the linear functional 
$$
L_u:\cV\ni v\mapsto(u|v)_{\cH}\in\bR\,,
$$
identifies $\cH$ with a dense subspace of $\cV'$.

\begin{lem}\lb{L-LemEnerg}
Assume that 
$$
v\in L^2(0,T;\cV)\quad\hbox{ and }\frac{dL_v}{dt}\in L^2(0,T;\cV')\,.
$$
Then

\noindent
(a) the function $v$ is a.e. equal to a unique element of $C([0,T],\cH)$ still denoted $v$;

\noindent
(b) this function $v\in C([0,T],\cH)$ satisfies
$$
\tfrac12|v(t_2)|^2_{\cH}-\tfrac12|v(t_1)|^2_{\cH}=\int_{t_1}^{t_2}\La\frac{dL_v}{dt}(t),v(t)\Ra_{\cV',\cV}dt
$$
for all $t_1,t_2\in[0,T]$
\end{lem}

Statement (a) follows from Proposition 2.1 and Theorem 3.1 in chapter 1 of \cite{LionsMage1}, and statement (b) from Theorem II.5.12 of \cite{FabrieBoyer}.

\begin{lem}\lb{L-Lempp}
Let $L\in L^2(0,T;\cV')$ satisfy
$$
\la L(t),w\ra_{\cV',\cV}=0\hbox{ for a.e. }t\in[0,T]
$$
for all $w\in\cV$. Then
$$
L(t)=0\hbox{ for a.e. }t\in[0,T]\,.
$$
\end{lem}

\begin{proof}
Pick $\cN_w\subset[0,T]$ negligible such that $L$ is defined on $[0,T]\setminus\cN_w$ and
$$
\la L(t),w\ra_{\cV',\cV}=0\hbox{ for all }t\in[0,T]\setminus\cN_w\,.
$$
Let $\cD$ be a dense countable subset of $\cV$ and let
$$
\bar\cN:=\bigcup_{w\in\cD}\cN_w\,.
$$
For all $t\in[0,T]\setminus\bar\cN$, one has
$$
\la L(t),w\ra_{\cV',\cV}=0\hbox{ for all }w\in\cD\quad\hbox{ so that }L(t)=0
$$
because $L(t)$ is a continuous linear functional on $\cV$ and $\cD$ is dense in $\cV$.
\end{proof}

\smallskip
The next lemma recalls the functional background for Green's formula in the context of evolution equations.

\begin{lem}\lb{L-NormTr}
Let $\Om$ be an open subset of $\bR^N$ with smooth boundary, and let $T>0$. Denote by $n$ the unit outward normal field on $\d\Om$. 
Let $\rho\in C([0,T];L^2(\Om))$, and let $m\in L^2((0,T)\times\Om,\bR^N)$. Assume that 
$$
\d_t\rho+\Div_xm=0\quad\hbox{ in the sense of distributions in }(0,T)\times\Om\,.
$$
Then

\noindent
(a) the vector field $m$ has a normal trace\footnote{We recall that $H^{1/2}_{00}((0,T)\times\d\Om)$ is the Lions-Magenes subspace of 
functions in $H^{1/2}((0,T)\times\d\Om)$ whose extension by $0$ to $\bR\times\d\Om$ defines an element 
of $H^{1/2}(\bR\times\d\Om)$; the notation $H^{1/2}_{00}((0,T)\times\d\Om)'$ designates the dual of $H^{1/2}_{00}((0,T)\times\d\Om)$.} 
$m\cdot n\rstr_{(0,T)\times\d\Om}\in H^{1/2}_{00}((0,T)\times\d\Om)'$;

\noindent
(b) for each $\psi\in H^1(\Om)$
$$
\ba
\frac{d}{dt}\int_\Om\rho(\cdot,x)\psi(x)dx&-\int_\Om m(\cdot,x)\cdot\grad_x\psi(x)dx
\\
&=-\la m\cdot n\rstr_{\d\Om},\psi\rstr_{\d\Om}\ra_{H^{-1/2}(\d\Om),H^{1/2}(\d\Om)}
\ea
$$
in $H^{-1}(0,T)$.
\end{lem}

\begin{proof}
Let $\chi\in C^\infty_c(\bR)$ be such that
$$
\chi(t)=1\hbox{ for }t\in[-1,T+1]\quad\hbox{ and }\Supp(\chi)\subset[-2,T+2]\,.
$$
Define
$$
\bar\rho(t,x):=\left\{\ba{}&\rho(t,x)\quad&&\hbox{ if }0\le t\le T\\&\chi(t)\rho(0,x)&&\hbox{ if }t<0\\&\chi(t)\rho(T,x)&&\hbox{ if }t>T\ea\right.
$$
and
$$
\bar m(t,x):=\left\{\ba{}&m(t,x)\quad&&\hbox{ if }0\le t\le T\\&0&&\hbox{ if }t\notin[0,T]\ea\right.
$$
so that the vector field $X:=(\bar\rho,\bar m)$ is an extension of $(\rho,m)$ to $\bR\times\Om$ satisfying
$$
X\in L^2(\bR\times\Om;\bR^{N+1})\,.
$$
Besides
$$
(\d_t\bar\rho+\Div_x\bar m)(t,x)=\chi'(t)(\indc_{t<0}\rho(0,x)+\indc_{t>T}\rho(T,x))=:S(t,x)
$$
with $S\in L^2(\bR\times\Om)$ so that
$$
\Div_{t,x}X=S\in L^2(\bR\times\Om)\,.
$$
Therefore $X$ has a normal trace on the boundary $\d(\bR\times\Om)=\bR\times\d\Om$, denoted $X\cdot n\rstr_{\bR\times\d\Om}\in H^{-1/2}(\bR\times\d\Om)$.

Let $\phi\in H^{1/2}_{00}((0,T)\times\d\Om)$; denote by $\bar\phi$ its extension by $0$ to $\bR\times\d\Om$. Thus $\bar\phi\in H^{1/2}(\bR\times\d\Om)$ and there exists 
$\bar\Phi\in H^1(\bR\times\Om)$ such that $\bar\phi=\bar\Phi\rstr_{\bR\times\d\Om}$. The normal trace of $m$ is then defined as follows: by Green's formula
$$
\ba
\la m\cdot n\rstr_{\bR\times\d\Om},\phi\ra_{H^{1/2}_{00}((0,T)\times\d\Om)'H^{1/2}_{00}((0,T)\times\d\Om)}
\\
:=
\la X\cdot n\rstr_{\bR\times\d\Om},\bar\phi\ra_{H^{1/2}((0,T)\times\d\Om)'H^{1/2}((0,T)\times\d\Om)}
\\
=
\iint_{\bR\times\Om}(\bar\rho\d_t\bar\Phi+\bar m\cdot\grad_x\bar\Phi+S\bar\Phi)(t,x)dxdt\,.
\ea
$$
Applying Green's formula on $(0,T)\times\Om$ shows that two different extensions  of the vector field $(\rho,m)$ define the same distribution $m\cdot n\rstr_{(0,T)}\times\d\Om)$ on 
$(0,T)\times\d\Om$. This completes the proof of statement a). 

As for statement b), let $\ka\in H^1_0(0,T)$ and $\psi\in H^1(\Om)$, define $\Phi(t,x):=\ka(t)\psi(x)$ and let $\bar\Phi$ be the extension of $\Phi$ by $0$ to $\bR\times\Om$, so that 
$\bar\Phi\in H^1(\bR\times\Om)$. Thus $\phi=\Phi\rstr_{(0,T)\times\d\Om}\in H^{1/2}_{00}((0,T)\times\d\Om)$ and
$$
\ba
{}&\la\la m\cdot n\rstr_{\d\Om},\psi\rstr_{\d\Om}\ra_{H^{-1/2}(\d\Om),H^{1/2}(\d\Om)},\ka\ra_{H^{-1}(0,T),H^1_0(0,T)}
\\
&\qquad\qquad:=
\la m\cdot n\rstr_{\bR\times\d\Om},\phi\ra_{H^{1/2}_{00}((0,T)\times\d\Om)'H^{1/2}_{00}((0,T)\times\d\Om)}
\\
&\qquad\qquad=
\iint_{\bR\times\Om}(\bar\rho\d_t\bar\Phi+\bar m\cdot\grad_x\bar\Phi+S\bar\Phi)(t,x)dxdt
\\
&\qquad\qquad=
\int_0^T\int_{\Om}(\rho(t,x)\ka'(t)\psi(x)+m(t,x)\cdot\grad\psi(x)\ka(t))dxdt
\\
&\qquad\qquad=
-\La\frac{d}{dt}\int_\Om\rho(t,x)\psi(x)dx,\ka\Ra_{H^{-1}(0,T),H^1_0(0,T)}
\\
&\qquad\qquad\quad\,\,+\int_0^T\int_{\Om}m(t,x)\cdot\grad\psi(x)\ka(t))dxdt
\ea
$$
which is precisely the identity in statement (b).
\end{proof}

\end{appendix}




\end{document}